\documentclass[12pt]{article}
\usepackage{comment}
\usepackage{mathrsfs}
\usepackage{amsmath}
\usepackage{amsfonts}
\usepackage{amssymb}
\usepackage{color}
\usepackage{amsmath, amsthm, amsfonts, amssymb}
\def\b{\beta}

\def \-{\bar}

\newtheorem{theorem}{Theorem}[section]
\newtheorem{lemma}[theorem]{Lemma}
\newtheorem{corollary}[theorem]{Corollary}
\newtheorem{proposition}[theorem]{Proposition}

\newtheorem{definition}[theorem]{Definition}

\newtheorem{remark}[theorem]{Remark}
\newtheorem{claim}[theorem]{Claim}

\date{}

\begin{document}

\title{\bf Complexity of holomorphic maps from the complex unit ball to classical domains} 

\author{Ming Xiao and Yuan Yuan\footnote{ Supported in
part by National Science Foundation grant DMS-1412384 and the seed grant program at Syracuse University}}

\vspace{3cm} \maketitle

\begin{abstract}
We study the complexity of holomorphic isometries and proper maps from the complex unit ball to type IV classical domains. We investigate on degree estimates of holomorphic isometries and holomorphic maps with minimum target dimension. We also construct a real-parameter family of mutually inequivalent holomorphic isometries from the unit ball to type IV domains.  We also provide examples of non-isometric proper holomorphic maps from the complex unit ball to classical domains.
\end{abstract}

\bigskip

\section{Introduction}

The motivation of this paper is twofold: the study of both isometries and proper maps between bounded symmetric space. 
Let $D, \Omega$ be bounded symmetric domains equipped with the Bergman metrics $\omega_{D}, \omega_{\Omega}$ respectively.  A holomorpic map $F: D \rightarrow \Omega $ is called isometric if $F^*\omega_{\Omega} = \lambda \omega_D$ for a positive constant $\lambda$. One recent attention to holomorphic isometries between bounded symmetric domains was paid by Clozel-Ullmo \cite{CU} in their study of the modular correspondence and later the holomorphic isometry problem was generalized extensively by Mok. He \cite{M5} proved a holomorphic isometry $F$ is totally geodesic when $D$ is irreducible and of rank at least 2. When $D$ is the complex unit ball of dimension at least 2 and $\Omega$ is the product of complex unit balls, $F$ is also totally geodesic \cite{YZ}. Much less is understood otherwise. When $D$ is the complex unit ball and $\Omega$ is an irreducible bounded symmetric domain of rank at least 2, a surprising non-totally geodesic phenomenon was discovered by Mok \cite{M6}. More precisely, for each irreducible $\Omega$, there exists a positive integer $n_\Omega$ such that $\mathbb{B}^{n_\Omega}$ admits a non-totally geodesic holomorphic isometry to $\Omega$ \cite{M6}. Moreover, if $\mathbb{B}^n$ admits a holomorphic isometry to $\Omega$, then $n \leq n_\Omega$ \cite{M6}. More recently, Chan-Mok characterize the image of the complex unit ball in $\Omega$ under the holomorphic isometry \cite{CM}.  
When $\Omega$ is the type IV classical domain, similar classification results  are obtained independently in \cite{CM}, \cite{UWZ}, \cite{XY}. The related problems on holomorphic isometries or holomorphic maps preserving invariant forms in Hermitian symmetric spaces are considered in \cite{C}, \cite{M2}, \cite{Ng1}, \cite{Ng2}, \cite{M5}, \cite{MN1}, \cite{MN2}, \cite{HY1}, \cite{HY2}, \cite{Ch}, \cite{Eb2} \cite{FHX}, \cite{Yu}, et al.

Proper holomorphic maps between bounded symmetric domains have also been a central subject in analysis and geometry of several complex variables. Let $F: D\rightarrow \Omega$ be a proper holomorphic map. When $D$ is of rank at least 2, many rigidity and non-rigidity results have been obtained in \cite{TH}, \cite{Ts}, \cite{Tu1}, \cite{Tu2}, \cite{M3}, \cite{KZ1}, \cite{KZ2}, \cite{Ng3}, et al. When $D=\mathbb{B}^n, \Omega=\mathbb{B}^N$, the rigidity and complexity of proper holomorphic maps remains a rather active problem in several complex variables.  Since Poincar\'e's pioneer work \cite{P}, many experts made fundamental progresses along the line (See \cite{Al}, \cite{L}, \cite{Fo}, \cite{St}, \cite{W},  \cite{Fa}, \cite{CS}, \cite{D1}, \cite{Hu1}, \cite{Eb1} and many references therein).  Roughly speaking, if the difference between $n, N$ is small, $F$ can be fully classified with additional assumptions on the boundary regularity of the map (cf. \cite{Al}, \cite{Fa},  \cite{Hu1}, \cite{HJ}, \cite{Ha}). In general, for arbitrary $n, N$, the moduli space of proper maps from $\mathbb{B}^n$ to $\mathbb{B}^N$ is rather large (cf.  \cite{CD}, \cite{DL2}).  

A gap phenomenon was then discovered in \cite{HJX} and later a gap conjecture was formulated by Huang-Ji-Yin \cite{HJY1} (see \cite{HJY2} as well). A proper holomorphic map $F: \mathbb{B}^n \rightarrow \mathbb{B}^N$ is called minimum if it cannot be reduced
to a map $(G,0)$  modulo automorphisms of $\mathbb{B}^n$ and $\mathbb{B}^N$ where $G$ has smaller target dimension.
The gap conjecture predicts precisely the intervals of $N$ where there are no minimum holomorphic proper maps from $\mathbb{B}^n$ to $\mathbb{B}^N$. They also showed that when $N$ is not in these intervals, then there is always a minimum monomial proper map from $\mathbb{B}^n$  to $\mathbb{B}^N$ \cite{HJY1}. These intervals terminate when the target dimension gets too large. (See als the work by D'Angelo and Lebl [DL1]). 

The authors proved  in a recent paper \cite{XY} that when $D=\mathbb{B}^n, \Omega= D^{IV}_m$ and the codimension is small,  any proper holomorphic map is indeed an isometry with additional boundary regularity assumptions.
In this paper, we continue to study the complexity of holomorphic proper and isometric maps from the complex unit ball to an irreducible classical domain. Motivated by the gap conjecture  mentioned above, we study  holomorphic proper (resp. isometric) maps $F: \mathbb{B}^n \rightarrow D^{IV}_m$  with minimum target dimension (See Section 3 for the precise definition).  We show in Section 3 that holomorphic isometries from $\mathbb{B}^n$ to $D_m^{IV}$ are {\it not} minimum if $m \geq 2n+3$. We also illustrate this result is optimal by constructing minimum holomorphic isometries from $\mathbb{B}^n$ to $D_m^{IV}$ for each $n+2 \leq m \leq 2n+2.$ On the other hand,  we prove that there always exist minimum proper holomorphic maps from $\mathbb{B}^n$ to $D_{m}^{IV}$ for any $m \geq n+1.$ In Section 4, we prove that there exists a real-parameter family of mutually inequivalent minimum holomorphic isometries (thus proper holomorphic maps) from $\mathbb{B}^n$ to $D_m^{IV}$ if $n+2 \leq  m \leq 2n+2$.  
Section 5 is devoted to establishing degree estimates for holomorphic isometric maps from $\mathbb{B}^n$ to $D_m^{IV}.$
In Section 6, we construct non-totally geodesic proper holomorphic maps from $\mathbb{B}^n$ to $\Omega$ where $\Omega$ is any of the  four types of classical domains. These maps can be non-isometric when $n$ is small while they become isometries when $n=n_\Omega$.
Interestingly, these examples further provide polynomial  proper holomorphic maps from $\mathbb{B}^n$ to $\Omega$ that answer a question of Mok (see also the independent work of Chan-Mok \cite{CM}).

A large part of the paper was finished in the December of 2015, before we learned of many interesting results on holomorphic isometries from $\mathbb{B}^n$ to $D^{IV}_m$ proved by Chan-Mok \cite{CM}. By combining the theorems of Chan-Mok \cite{CM} and ours \cite{XY} (see the similar result in \cite{UWZ} as well), it is clear that any holomorphic isometry $F: \mathbb{B}^n \rightarrow D^{IV}_m$ can be written in the following form $F= \varphi \circ f \circ \tau \circ i \circ \sigma$, where $\sigma \in Aut(\mathbb{B}^n), \tau \in Aut(\mathbb{B}^{m-1}), \varphi \in Aut(D^{IV}_{m})$, $i$ is the standard linear embedding from $\mathbb{B}^n$ to  $\mathbb{B}^{m-1}$ and $f$ is either $R^{IV}_{m-1}$ or $I^{IV}_{m-1},$ which are defined in Section 6.4. The study on minimality and inequivalent families of holomorphic isometries from the complex unit ball to the type IV classical domains in Section 3 and Section 4  is motivated by the analogue study for proper holomorphic maps between balls.

\bigskip

{\bf Acknowledgement}: The authors are grateful to Professor J. D'Angelo,  Professor X. Huang,  Professor N. Mok, Dr. S. T. Chan and Dr. S. Ng for helpful discussions.

\section{Preliminaries}
Irreducible bounded symmetric domains are realized as Cartan's four types of domains and two exceptional cases (cf. \cite{H2} \cite{M1}). Assume $q \geq p$ and let $M(p, q; \mathbb{C})$ denote the space of $p \times q$ matrices with entries of complex numbers. The type I domain is defined as

$$D^I_{p, q} = \{Z \in M(p, q; \mathbb{C}) : I_p -  Z \overline{Z}^t>0 \}.$$ In particular, when $p=1$, the type I domain is the complex unit ball $\mathbb{B}^q = \{z=(z_1, \cdots, z_q) \in \mathbb{C}^q : |z|^2 <1\}$ in $\mathbb{C}^q$. The type II and type III
domains are submanifolds of $D^I_{n, n}$ defined as

$$D^{II}_n = \{Z \in D^I_{n, n} : Z = - Z^t\}$$
and $$D^{III}_n = \{Z \in D^I_{n, n} : Z = Z^t\}.$$
The type IV domain is defined as $$D^{IV}_n = \{Z =(z_1, \cdots, z_n) \in \mathbb{C}^n : Z \overline{Z}^t <2 ~\text{and} ~ 1- Z \overline{Z}^t + \frac{1}{4} |Z Z^t|^2 >0 \}.$$

Let $\Omega$ be an irreducible classical domain.
The Bergman kernel function $K_\Omega(Z, \bar Z)$ is explicitly given by
\begin{equation}
\begin{split}
K_\Omega(Z, \bar Z) &= c_I \left( \det(I_p - Z\overline{Z}^t ) \right)^{-(p+q)} ~\text{when}~\Omega = D^I_{p, q};  \\
K_\Omega(Z, \bar Z) &= c_{II} \left( \det(I_n - Z\overline{Z}^t ) \right)^{-(n-1)} ~\text{when}~\Omega = D^{II}_n; \\
K_\Omega(Z, \bar Z) &= c_{III} \left( \det(I_n -  Z\overline{Z}^t) \right)^{-(n+1)} ~\text{when}~\Omega = D^{III}_n; \\
K_\Omega(Z, \bar Z) &= c_{IV} \left( 1 -  Z\overline{Z}^t +\frac{1}{4} |Z Z^t|^2 \right)^{-n} ~\text{when}~\Omega = D^{IV}_n ,\\
\end{split}
\end{equation}
where $c_*$ are positive constants depending on $n$ and the type of $\Omega$ (cf. \cite{H2} \cite{M1}). The Bergman metric 
$$\omega_\Omega(Z) :=\sqrt{-1} \partial \bar\partial \log K_\Omega(Z, \bar Z)$$ on $\Omega$ is K\"ahler-Einstein as the Bergman kernel function is invariant under holomorphic automorphisms.
%
%
Note that the standard linear embedding $L(Z) =Z$ from $D^{II}_{n}$ or $D^{III}_{n}$ into $D^{I}_{n, n}$ is a totally geodesically holomorphic isometric embedding with respect to Bergman metrics with isometric constant $\frac{2n}{n-1}$ or $\frac{2n}{n+1}$ respectively.

Let $S$ be the Hermitian symmetric space of compact type dual to  $\Omega$ and $\delta \in H^2(S, \mathbb{Z})$ be the positive generator. It is well-known that the first Chern class $c_1(S) = (p+q) \delta, 2(n-1) \delta, (n+1)\delta$ or $n \delta$, When $\Omega= D^{I}_{p,q}, D^{II}_n, D^{III}_{n}, D^{IV}_n$ respectively.
Therefore, it follows from Mok's theorem in \cite{M6} that $n_\Omega= p+q-1, 2n-3, n$ or $n-1$ when the classical domain $\Omega  = D^I_{p, q}, D^{II}_{n}, D^{III}_{n}$ or $D^{IV}_n$ respectively.

We now describe the holomorphic automorphism group action on $D^{IV}_m$ in terms of the Borel embedding (cf. \cite{H1} \cite{M1}).
The hyperquadric $\mathbb{Q}^m $, the compact dual of $D^{IV}_n$ is defined by $\mathbb{Q}^m := \{[z_1, \cdots, z_{m+2}] \in \mathbb{P}^{m+1} | z_1^2+ \cdots + z_m^2 = z^2_{m+1} + z_{m+2}^2 \}$.
The Borel embedding $D^{IV}_m \subset \mathbb{Q}^m \subset \mathbb{P}^{m+1}$ is given by $$Z= (z_1, \cdots, z_m) \rightarrow \left[z_1, \cdots, z_m, \frac{1+ \frac{1}{2}Z Z^t}{\sqrt{2}}, \frac{1- \frac{1}{2}Z Z^t}{\sqrt{-2}}\right].$$
The holomorphic automorphism group of $D^{IV}_m$ is given by 
\begin{equation}\notag
{\rm Aut}(D^{IV}_{m}) = \left\{
\begin{bmatrix}
A & B\\
C & D\\
\end{bmatrix}  \in O(m, 2, \mathbb{R}) |  {\rm det}(D)>0
\right\},
\end{equation}
where $A\in M(m, m, \mathbb{R}), B \in M(m, 2, \mathbb{R}),
 C \in M(2, m, \mathbb{R}), D \in M(2, 2, \mathbb{R})$.
The automorphism group action is given in the following explicit way. Let $Z = (z_1, \cdots, z_m) \in D^{IV}_m$ and  $T =\begin{bmatrix}
A & B\\
C & D\\
\end{bmatrix} \in {\rm Aut}(D^{IV}_{n}).$ Write $Z'=\left(\frac{1+ \frac{1}{2}Z Z^t}{\sqrt{2}}, \frac{1- \frac{1}{2}Z Z^t}{\sqrt{-2}}\right)$. Then the action of $T$ on $D^{IV}_m$ is  given by  $$T (Z) = \frac{Z A +  Z' C}{\left( Z B + Z' D \right) \left(1/\sqrt{2}, \sqrt{-1/2}\right)^t}.$$
Rephrasing in homogenous coordinates, if the holomorphic automorphism maps $Z=(z_1, \cdots, z_m) \in D^{IV}_m$ to $W=(w_1, \cdots, w_m) \in D^{IV}_m$, then there exists $T \in {\rm Aut}(D^{IV}_{m})$ such that 
$$\left[w_1, \cdots, w_m, \frac{1+\frac{1}{2}WW^t}{\sqrt{2}},  \frac{1-\frac{1}{2}WW^t}{\sqrt{-2}}\right] = \left[ z_1, \cdots, z_m,  \frac{1+\frac{1}{2}ZZ^t}{\sqrt{2}},  \frac{1-\frac{1}{2}ZZ^t}{\sqrt{-2}}\right] \cdot T.$$
 In other words, there exists nonzero $\lambda \in \mathbb{C}$, such that 
 $$\left(w_1, \cdots, w_m, \frac{1+\frac{1}{2}WW^t}{\sqrt{2}},  \frac{1-\frac{1}{2}WW^t}{\sqrt{-2}} \right) = \lambda \left(z_1, \cdots, z_m,  \frac{1+\frac{1}{2}ZZ^t}{\sqrt{2}},  \frac{1-\frac{1}{2}ZZ^t}{\sqrt{-2}}\right)\cdot T.$$ 
Note that the isotropy group $K_0$ at the
origin is $K_0 = \left\{\begin{bmatrix}
A & 0\\
0 & D\\
\end{bmatrix} \in O(m, 2, \mathbb{R}) | {\rm det}(D)=1 \right\} \cong O(m, \mathbb{R}) \times SO(2, \mathbb{R}).$
\section{Holomorphic maps from $\mathbb{B}^n$ to $D^{IV}_m$}
 In this section, we study the holomorphic maps from $\mathbb{B}^n$ to $D^{IV}_m$ with 
minimum dimension in the target. We say two holomorphic maps $F_1, F_2: \mathbb{B}^n \rightarrow D_{m}^{IV}$ are equivalent if $F_1=\phi \circ F_2 \circ \psi,$ where $\psi, \phi$ are automorphisms of $\mathbb{B}^n$ and $D_{m}^{IV}$, respectively.
A holomorphic map $F: \mathbb{B}^n \rightarrow D^{IV}_m$  is called  minimum if there is no holomorphic map $G: \mathbb{B}^n \rightarrow D^{IV}_{m-1}$ such that $F$ is equivalent to $(G, 0)$.

\bigskip

\subsection{Minimum of holomorphic isometries to Type IV domains}
We study in this subsection the minimum holomorphic isometric maps from $\mathbb{B}^n$ to $D_{m}^{IV}.$ We first note that it follows from Mok's theorem \cite{M6} that any holomorphic isometry $F: \mathbb{B}^n \rightarrow D^{IV}_{n+1}$ is minimum. We prove the following theorem that there are no minimum holomorphic isometric maps from $\mathbb{B}^n$ to $D_m^{IV}$ if $m > 2n+2.$

\begin{theorem}
Let $m > 2n+2, n \geq 2$. Let $F: \mathbb{B}^n \rightarrow D^{IV}_m$ be a holomorphic isometry. Then there exists a holomorphic isometry $G: \mathbb{B}^n \rightarrow D^{IV}_{2n+2}$ such that $F$ is equivalent to $(G, \bf 0)$, where $\bf 0$ is a $(m-2n-2)$-dimensional zero row vector.
\end{theorem}

\begin{proof}
It suffices to show that for any such $F$ and $m$, there exists a holomorphic isometry $\hat F: \mathbb{B}^n \rightarrow D^{IV}_{m-1}$ such that $F$ is equivalent to $(\hat F, \bf 0)$. By the isometric assumption, we have 
\begin{equation}\label{eqnismi}
F^{*}(\omega_{D_m^{IV}})=\lambda \omega_{\mathbb{B}^n},
\end{equation}
for some positive constant $\lambda.$
From Proposition 2.11 in \cite{XY}, we know that $\lambda$ must be $m/(n+1).$ Write $Z=(z_1,...,z_n)$ as the coordinates in $\mathbb{C}^n.$
Write $F=(F_1, \cdots, F_m)$. By composing with the autmorphism of $D_m^{IV}$ if necessary, we may assume $F(0)=0.$  By standard reduction, we obtain from (\ref{eqnismi}),
$$1-F\overline{F}^t+\frac{1}{4}|FF^t|^2=1-Z\overline{Z}^t.$$
By a lemma of D'Angelo (\cite{D2}), we have
$$(F_1, \cdots, F_m)= (z_1, \cdots, z_n, \frac{1}{2} \sum_{i=1}^m F^2_i(z), 0, \cdots, 0) \cdot \bf V ,$$
where $\bf V$ is an $m \times m$ unitary matrix. Write $$V =\begin{bmatrix} \bf v_1 \\ \cdots \\ \bf v_m \end{bmatrix},$$ where $\bf v_i$ is a $m$-dimensional row vector for $1 \leq i \leq m$. Then we have
$$\left(z_1, \cdots, z_n, \frac{1}{2} \sum_{i=1}^m F^2_i(z)\right) \cdot \begin{bmatrix} \bf v_1 \\ \cdots \\ \bf v_{n+1} \end{bmatrix} = (F_1, \cdots, F_m).$$

{\bf Claim:} $\{F_1, \cdots, F_m\}$ is a linearly dependent set over real number field. In other words,
there exist $\lambda_1, \cdots, \lambda_m \in \mathbb{R}$ not mutually zero, such that $\sum_{i=1}^m \lambda_i F_i  \equiv 0$.

\medskip

{\bf Proof of Claim:} Write ${\bf v_i} ={\bf a_i} + \sqrt{-1} {\bf b_i}$ for $1 \leq i \leq n+1$. It is easy to see that there exists ${\bf v} = (\lambda_1, \cdots, \lambda_m)^t \in \mathbb{R}^m$ with ${\bf v} \not= 0$ such that
\begin{equation}\label{rank}
\begin{bmatrix}
\bf a_1 \\ \bf b_1 \\ \cdots \\ \bf a_{n+1} \\ \bf b_{n+1}
\end{bmatrix} {\bf v} = \bf 0.
\end{equation}
This is because $$\text{rank} \begin{bmatrix}
\bf a_1 \\ \bf b_1 \\ \cdots \\ \bf a_{n+1} \\ \bf b_{n+1}
\end{bmatrix} \leq 2(n+1) < m.$$
Then (\ref{rank}) implies $$\left(z_1, \cdots, z_n, \frac{1}{2} \sum_{i=1}^m F^2_i(z)\right) \cdot \begin{bmatrix} \bf v_1 \\ \cdots \\ \bf v_{n+1} \end{bmatrix} \cdot {\bf v} = \bf 0.$$
This proves the claim by showing $\sum_{i=1}^m \lambda_i F_i  \equiv 0$.

\medskip

By rescaling ${\bf v}$ if necessary, we assume $|{\bf v}| =1$. Extend ${\bf v}$ to an orthonormal basis $\{\bf u_1, \cdots, u_{m-1}, {\bf v} \}$ of $\mathbb{R}^m$ and write $m\times m$ matrix $\bf C = (\bf u_1, \cdots, u_{m-1}, {\bf v})$. Define $\hat F = (\hat F_1, \cdots, \hat F_m) = F \cdot \bf C$ and then $\hat F$ is equivalent to $F$. This completes the proof of the theorem because $\hat F_m = F \cdot {\bf v} =0$.
\end{proof}

\medskip

Define $R_{n+2}, I_{n+2}: \mathbb{B}^n \rightarrow D^{IV}_{n+2}$ to be
$$ R_{n+2}(z)
= \bigg( \cos{\theta_1} z_1, \sqrt{-1}\sin{\theta_1} z_1, z_2, \cdots, z_{n-1}, $$
$$ \frac{\cos(2{\theta_1}) z_1^2 + \sum_{j=2}^{n-1} z_j^2 -2 z_n^2 +2 z_n}{2\sqrt{2}(1-z_n)}, \frac{\cos(2{\theta_1}) z_1^2 + \sum_{j=2}^{n-1} z_j^2 +2 z_n^2 -2 z_n}{2\sqrt{-2}(1-z_n)} \bigg), $$
$$ I_{n+2}(z)
= \left(\cos{\theta_1} z_1, \sqrt{-1}\sin{\theta_1} z_1, z_2, \cdots, z_{n}, 1-\sqrt{1-\cos(2{\theta_1}) z^2_1 - \sum_{j=2}^{n} z_j^2} \right) $$
with $\theta \in (0, \pi/4]$;
and one can similarly define $R_{n+k}, I_{n+k}: \mathbb{B}^n \rightarrow D_{n+k}^{IV}$ for $2 \leq k \leq n:$
$$ R_{n+k}(z)
= \bigg( \cos{\theta_1} z_1, \sqrt{-1}\sin{\theta_1} z_1, \cdots, \cos{\theta_{k-1}}z_{k-1}, \sqrt{-1}\sin{\theta_{k-1}}z_{k-1},$$
$$z_k, \cdots, z_{n-1}, \frac{\sum_{j=1}^{k-1}\cos(2{\theta_j}) z_j^2+\sum_{j=k}^{n-1} z_j^2  -2 z_n^2 +2 z_n}{2\sqrt{2}(1-z_n)}, \frac{\sum_{j=1}^{k-1}\cos(2{\theta_j}) z_j^2+\sum_{j=k}^{n-1} z_j^2 +2 z_n^2 -2 z_n}{2\sqrt{-2}(1-z_n)} \bigg) $$
 $$I_{n+k}(z)= \bigg( \cos{\theta_1} z_1, \sqrt{-1}\sin{\theta_1} z_1, \cdots, \cos{\theta_{k-1}}z_{k-1}, \sqrt{-1}\sin{\theta_{k-1}}z_{k-1}, $$
$$z_k, \cdots, z_n, 1-\sqrt{1-\sum_{j=1}^{k-1}\cos(2{\theta_j}) z^2_j - \sum_{j=k}^ n z_j^2}  \bigg)$$
with $\theta_j \in (0, \pi/4]$ for $1 \leq j \leq k-1$. Here when $k=n,$ the components $``z_k, \cdots ,z_{n-1}"$ in $R_{n+k}$ is understood to be void. Furthermore, $R_{2n+1}, I_{2n+1}: \mathbb{B}^n \rightarrow D^{IV}_{2n+1}$ are given by
$$R_{2n+1}(z)= \bigg( \cos{\theta_1} z_1, \sqrt{-1}\sin{\theta_1} z_1, \cdots, \cos{\theta_{n-1}}z_{n-1}, \sqrt{-1}\sin{\theta_{n-1}}z_{n-1}, z_n, $$
$$\frac{1}{2\sqrt{2}} \left( \sum_{j=1}^{n-1} \cos(2\theta_j) z_j^2 + z^2_n \right), \frac{-\sqrt{-1}}{2\sqrt{2}}\left( \sum_{j=1}^{n-1} \cos(2\theta_j) z_j^2 + z^2_n \right) \bigg)$$
for $\theta_1, \cdots, \theta_{n-1} \in (0, \pi/4]$,
$$I_{2n+1}(z)= \left( \cos{\theta_1} z_1, \sqrt{-1}\sin{\theta_1} z_1, \cdots, \cos{\theta_{n}}z_{n}, \sqrt{-1}\sin{\theta_{n}}z_{n}, 1-\sqrt{1-\sum_{j=1}^{n}\cos(2{\theta_j}) z^2_j} \right)$$ for $\theta_1, \cdots, \theta_{n} \in (0, \pi/4]$ but not all $\theta_j = \pi/4$; and $R_{2n+2}, I_{2n+2}: \mathbb{B}^n \rightarrow D^{IV}_{2n+2}$ are given by
$$R_{2n+2}(z) = \bigg( \cos{\theta_1} z_1, \sqrt{-1}\sin{\theta_1} z_1, \cdots, \cos{\theta_{n}}z_{n}, \sqrt{-1}\sin{\theta_{n}}z_{n}, $$
$$\frac{1}{2\sqrt{2}}  \sum_{j=1}^{n} \cos(2\theta_j) z_j^2  , \frac{-\sqrt{-1}}{2\sqrt{2}} \sum_{j=1}^{n} \cos(2\theta_j) z_j^2  \bigg),$$
$$I_{2n+2}(z) = \bigg( \cos{\theta_1} z_1, \sqrt{-1}\sin{\theta_1} z_1, \cdots, \cos{\theta_{n}}z_{n}, \sqrt{-1}\sin{\theta_{n}}z_{n}, $$
$$\frac{\cos\theta}{\cos(2\theta)} \left(1 -  \sqrt{1-\cos(2\theta) \left(\sum_{j=1}^n \cos(2\theta_j)z_j^2\right)}\right), \frac{\sqrt{-1}\sin\theta}{\cos(2\theta)} \left(1 -  \sqrt{1-\cos(2\theta) \left(\sum_{j=1}^n \cos(2\theta_j)z_j^2\right)}\right) \bigg)$$ for $\theta _j \in (0, \pi/4]$ but not all $\theta_j = \pi/4$ for $1 \leq  j \leq n$ and $\theta \in (0, \pi/4)$.

\begin{theorem}\label{thmghminimal}
For each $2 \leq k \leq n+2$, $R_{n+k}, I_{n+k}: \mathbb{B}^n \rightarrow D^{IV}_{n+k}$ are minimum holomorphic isometries.
\end{theorem}

\begin{proof}
We will only prove the case $k=2$ and other cases follow by similar argument.
Apply the Borel embedding to embed $\mathbb{B}^n$ as an open subset of $\mathbb{P}^n$ and $D^{IV}_{m}$ as an open subset of $\mathbb{Q}^{m} \subset \mathbb{P}^{m+1}$ and write $[z, s] = [z_1, \cdots, z_n, s]$ to denote the homogeneous coordinates in $\mathbb{P}^n$. Here recall the Borel embedding of $\mathbb{B}^n$ into $\mathbb{P}^n$ is given by
$$(z_1, \cdots,z_n) \rightarrow [z_1, \cdots, z_{n},1].$$
The Borel embedding of $D_{m}^{IV}$ into $\mathbb{Q}^m \subset \mathbb{P}^{m+1}$ is as described in Section 2.

We first prove the theorem for $R_{n+2}$. Under the homogeneous coordinates, $R_{n+2}$ can be identified with
$$\mathcal{R}_{n+2}(z, s)
= \left[g_1(z, s), \cdots, g_{n+4}(z, s)  \right]
$$ from $\mathbb{P}^n$ to $\mathbb{P}^{n+3}$, where
\begin{equation}\notag
\begin{split}
g_1(z, s) &= \cos\theta_1 (s-z_n) z_1; \\
g_2(z, s) &= \sqrt{-1} \sin\theta_1 (s-z_n) z_1; \\
g_j(z, s) &= (s-z_n) z_{j-1}  {\rm~for~} 3 \leq j \leq n ; \\
g_{n+1}(z, s) &=  \frac{\cos(2{\theta_1}) z_1^2 + \sum_{j=2}^{n-1} z_j^2 -2 z_n^2 +2 z_n s}{2\sqrt{2}};\\
g_{n+2}(z, s) &= \frac{\cos(2{\theta_1}) z_1^2 + \sum_{j=2}^{n-1} z_j^2 +2 z_n^2 -2 z_n s}{2\sqrt{-2}};\\
g_{n+3}(z, s) &= \frac{1}{2\sqrt{2}} \left( \cos(2\theta_1) z_1^2 + \sum_{j=2}^{n} z^2_j +2 s^2 - 2 z_n s \right); \\
g_{n+4}(z, s) &= \frac{1}{2\sqrt{-2}} \left( -\cos(2\theta_1) z_1^2 - \sum_{j=2}^{n} z^2_j +2 s^2 - 2 z_n s \right).
\end{split}
\end{equation}

{\bf Claim:} The set $\{g_1, \cdots, g_{n+4}\}$ is linearly independent over $\mathbb{R}$ on any open subset of $\mathbb{C}^{n+1}$. Consequently, for any ${\bf B} \in U(n, 1)$, the set $\{\hat{g}_1, \cdots, \hat{g}_{n+4}\}$ with $\hat g_{j} = g_j((z, s)\cdot \bf B)$ is linearly independent over $\mathbb{R}$.
\medskip

{\bf Proof of Claim:} We will just prove the first part of the claim and the second part follows easily. Let $\{a_1, \cdots, a_{n+4}\}$ be a set of real numbers such that $\sum_{j=1}^{n+4} a_j \hat g_j \equiv 0$. By comparing coefficients of $z_n z_j$ for $1 \leq j \leq n-1$, we know $a_j =0$ for $1 \leq j \leq n$. By comparing coefficients of $z_n^2$, we know $a_{n+1}=a_{n+2} =0$. By comparing coefficients of $s^2$, we know $a_{n+3}=a_{n+4} =0.$ This proves the claim.

\medskip

Now suppose that $R_{n+2}$ is not minimum. Namely, there exists $F: \mathbb{B}^n \rightarrow D^{IV}_m$ with $m < n+2$ such that $R_{n+2}$ is equivalent to $(F, 0)$. More precisely, under homogeneous coordinates, there exist ${\bf B} \in U(n, 1)$ and $T \in {\rm Aut}(D^{IV}_{n+2})$ such that \begin{equation}\label{contra}
\mathcal{R}_{n+2}((z, s) \cdot {\bf B}) \cdot T = [\tilde{F}(z, s), 0, \cdots],
\end{equation}
where $\tilde{F}$ is the map obtained from $F$ under homogeneous coordinates. By comparing the $(n+2)$-th element in (\ref{contra}), we deduce a contradiction to the claim. This shows that $R_{n+2}$ must be minimum.

The conclusion for $I_{n+2}$ in the theorem follows from the similar argument. Under the homogeneous coordinates, $I_{n+2}$ can be identified with
$$\mathcal{I}_{n+2}(z, s)
= \left[h_1(z, s), \cdots, h_{n+4}(z, s)  \right]
$$ from $\mathbb{P}^n$ to $\mathbb{P}^{n+3}$, where
\begin{equation}\notag
\begin{split}
h_1(z, s) &= \cos\theta_1 z_1; \\
h_2(z, s) &= \sqrt{-1} \sin\theta_1 z_1; \\
h_j(z, s) &= z_{j-1}  {\rm~for~} 3 \leq j \leq n+1 ; \\
h_{n+2}(z, s) &= s - \sqrt{s^2 - \cos(2\theta_1) z_1^2 - \sum_{j=2}^n z_j^2};\\
h_{n+3}(z, s) &= \frac{1}{\sqrt{2}} \left( 2s - \sqrt{s^2 - \cos(2\theta_1) z_1^2 - \sum_{j=2}^n z_j^2} \right); \\
h_{n+4}(z, s) &= \frac{1}{\sqrt{-2}} \sqrt{s^2 - \cos(2\theta_1) z_1^2 - \sum_{j=2}^n z_j^2}.
\end{split}
\end{equation}

{\bf Claim:} The set $\{h_1, \cdots, h_{n+4}\}$ is linearly independent over $\mathbb{R}$ on any open subset of $\mathbb{C}^{n+1}$. Consequently, for any ${\bf B} \in U(n, 1)$, the set $\{\hat{h}_1, \cdots, \hat{h}_{n+4}\}$ with $\hat h_{j} = h_j((z, s)\cdot \bf B)$ is linearly independent over $\mathbb{R}$.
\medskip

The proof the claim is very similar to the previous one. Let $\{a_1, \cdots, a_{n+4}\}$ be a set of real numbers such that $\sum_{j=1}^{n+4} a_j \hat h_j \equiv 0$. Then one can show $a_j=0$ for all $j$ by comparing coefficients. 
\medskip

The rest proof of the theorem is also similar. Suppose that $I_{n+2}$ is not minimum. Namely, there exists $F: \mathbb{B}^n \rightarrow D^{IV}_m$ with $m < n+2$ such that $I_{n+2}$ is equivalent to $(F, 0)$. More precisely, under homogeneous coordinates, there exist ${\bf B} \in U(n, 1)$ and $T \in {\rm Aut}(D^{IV}_{n+2})$ such that
\begin{equation}\label{contra2}
\mathcal{I}_{n+2}((z, s) \cdot {\bf B}) \cdot T = (\tilde{F}(z, s), 0, \cdots),
\end{equation}
where $\tilde{F}$ is the map obtained from $F$ under homogeneous coordinates. By comparing the $(n+2)$-th element in (\ref{contra2}), we deduce a contradiction to the claim. This shows that $I_{n+2}$ must be minimum.
\end{proof}

\subsection{Minimum holomorphic proper maps to Type IV domains}
We investigate minimim holomorphic proper maps from $\mathbb{B}^n$ to $D_m^{IV}$ in this subsection. By Lemma 2.2 in \cite{XY}, there is no proper holomorphic maps from $\mathbb{B}^n$ to $D_{m}^{IV}$ if $n \geq 2, m \leq n,.$  The following theorem reveals a different phenomenon of proper holomorphic maps from isometries. 

\begin{theorem}\label{thmminimum}
For any $m \geq n+1 \geq 2,$ there is a minimum proper holomorphic map from $\mathbb{B}^n$ to $D_{m}^{IV}.$
\end{theorem}

To establish Theorem \ref{thmminimum}, we first prove the following result.
\begin{theorem}\label{thmpminimum}
Let $N \geq n \geq 1.$ Let $F=(f_1, \cdots, f_N)$  be  a minimum monomial proper map from $\mathbb{B}^n$ to $\mathbb{B}^N$, where each $f_i=c_i z^{\alpha_i}$ for some multiindex $\alpha_i$ and 
 real number $c_i.$ Define holomorpic  maps from $\mathbb{B}^n$ to $\mathbb{C}^{N+k}, 1 \leq k \leq N+1$ associated to $F$ as follows.
$$M^{N+1}_F:=\left( f_1,\cdots, f_N, 1-\sqrt{1-\sum_{i=1}^N f_i^2}\right),$$
and for each $2 \leq k \leq N+1,$ fixing $\theta \in (0, \frac{\pi}{4}),$ define

$$M^{N+k}_F :=  \bigg( \cos{\theta} f_1, \sqrt{-1}\sin{\theta} f_1, \cdots, \cos{\theta}f_{k-1}, \sqrt{-1}\sin{\theta}f_{k-1}$$

$$f_k, \cdots, f_N, 1-\sqrt{1-\cos(2\theta) \sum_{j=1}^{k-1} f^2_j - \sum_{j=k}^ N f_j^2}  \bigg).$$
Then each $M^{N+k}_F, 1 \leq k \leq N+1,$ is a minimum proper holomorphic map from $\mathbb{B}^n$ to $D_{N+k}^{IV}.$
\end{theorem}

\begin{proof}
We will prove only the statement for $k=1.$ The other cases can be proved by a similar argument. Write $m=N+1, H=M_{F}^{N+1}.$
Denote the highest degree of $f_i, 1 \leq i \leq N$ by $d$.

By direct computation one can verify that $H$ is a proper holomorphic  map from $\mathbb{B}^n$ to $D_m^{IV}.$
We now prove that $H$ is minimum. For that we again apply the Borel embedding to embed $\mathbb{B}^n$ as an open subset of $\mathbb{P}^n$ and $D_m^{IV}$ as an open subset  $\mathbb{Q}^m \subset \mathbb{P}^{m+1}.$ Write $[z,s]=[z_1,...,z_n,s]$ to denote the homogeneous coordinates in $\mathbb{P}^n.$ Under the homogeneous coordinates, $H$ is identified with
$$\mathcal{H}(z,s)=[h_1(z,s),...,h_{m+2}(z,s)]$$
from $\mathbb{P}^{n}$ to $\mathbb{P}^{m+1},$ where
$$h_1(z,s)=s^d f_1\left(\frac{z}{s}\right),~~\cdots,~~ h_N(z,s)=s^d f_N\left(\frac{z}{s} \right),$$ 
$$h_m(z,s)=s^d-\sqrt{s^{2d}-\sum_{i=1}^N h_i^2},$$
$$h_{m+1}(z,s)=\frac{1}{\sqrt{2}}\left(2s^d-\sqrt{s^{2d}-\sum_{i=1}^N h_i^2} \right),$$
$$h_{m+2}(z,s)=\frac{1}{\sqrt{-2}}\sqrt{s^{2d}-\sum_{i=1}^N h_i^2}.$$

\begin{claim}\label{lemmadep}
The set $\{h_1,...,h_{m+2}\}$ is linearly independent over $\mathbb{R}.$ Consequently, for any ${\bf B} \in U(n,1),$ the set $\{\hat{h}_1, \cdots, \hat{h}_{m+2}\}$ with $\hat{h}_j(z, s)=h_j((z,s)\cdot {\bf B})$ is linearly independent over $\mathbb{R}.$
\end{claim}
 We first note that $s^{2d}-\sum_{i=1}^N h_i^2$ is not a perfect square of a polynomial. This amounts to the following easy lemma whose proof will be left to readers.
\begin{lemma}\label{xxxxxxx}
Let $f_i, 1 \leq i \leq N$, be mutually distinct monomials with real coefficients. Then  $1-\sum_{i=1}^N f_i^2$ is not a perfect square of a polynomial. 
\end{lemma}

It follows from Lemma \ref{xxxxxxx} that $\{h_1,...,h_N, s^d, \sqrt{s^{2d}-\sum_{i=1}^N h_i^2}\}$ is linearly independent over $\mathbb{R}.$ This further implies $\{h_1,...,h_{m+2}\}$ is also linearly independent over $\mathbb{R}$. The latter part of Claim \ref{lemmadep} is then an easy  consequence.
The rest of proof is just a copy of that of Theorem \ref{thmghminimal}.
\end{proof}

We recall the results about the gap conjecture on proper maps between balls. The following intervals appear in the gap conjecture for proper holomorphic maps between balls \cite{HJY1}.
For $n > 2$, let $K(n):= {\rm max} \{m \in
\mathbb{Z}^+: m(m+1)/2<n\}$ and let $I_k:= \{m\in \mathbb{Z}^+: kn<m<(k+1)n-k(k+1)/2\}$
for $1\leq k\leq K(n)$. The following theorem is proved by Huang-Ji-Yin (See also D'Angelo-Lebl [DL1]).


\begin{theorem}(\text{[HJY1]})\label{thmhjy}
For any $N\geq n \geq 2$ with $N\not\in \cup_{k=1}^{K(n)} I_k$, there is a minimum
proper monomial map from $\mathbb{B}^n$ to $\mathbb{B}^N.$ 
\end{theorem}

{\bf Proof of Theorem \ref{thmminimum}:} When $n=1,$ clearly there are minimum monomial proper maps from $\Delta$ to $\mathbb{B}^N$ for any $N \geq 1.$ When $n \geq 2,$ by Theorem \ref{thmhjy} there is always a minimum monomial proper map from $\mathbb{B}^n$ to $\mathbb{B}^{kn}$ for any $k \geq 1.$  Then Theorem \ref{thmminimum} is a consequence of Theorem \ref{thmpminimum}.

\section{Inequivalent families of holomorphic isometries}
Let $I_{n+k}$ be the maps defined in Section 3. In this section, we will make use of $I_{n+k}$ to give inequivalent families of holomorphic isometries from $\mathbb{B}^n$ to $D_{n+k}^{IV}, 2 \leq k \leq n+2.$ To emphasize the dependence on $\theta,$ we will write $I_{n+k, \theta}$ instead of $I_{n+k}.$ More precisely, for $\theta \in (0, \pi/4]$, define $$I_{n+2, \theta}(z)= \left(\cos{\theta} z_1, \sqrt{-1}\sin{\theta} z_1, z_2, \cdots, z_{n}, 1-\sqrt{1-\cos(2{\theta}) z^2_1 - \sum_{j=2}^{n} z_j^2} \right). $$
Fixing $\beta \in (0, \pi/4)$, for $\theta \in [\beta, \pi/4]$ define 
$$I_{n+3, \theta}(z)= \bigg( \cos{\theta} z_1, \sqrt{-1}\sin{\theta} z_1, \cos{\beta} z_2, \sqrt{-1}\sin{\beta} z_2, z_3, \cdots, z_{n}, $$
$$1-\sqrt{1-\cos(2{\theta}) z^2_1 -\cos 2\beta z_2^2 - \sum_{j=2}^{n} z_j^2} \bigg). $$
Similarly, we define $I_{n+k, \theta}: \mathbb{B}^n \rightarrow D_{n+k}^{IV}$ for all $3 \leq k \leq n+1:$ 

$$I_{n+k, \theta} = \bigg( \cos{\theta} z_1, \sqrt{-1}\sin{\theta} z_1, \cos{\beta} z_2, \sqrt{-1}\sin{\beta} z_2, \cdots, \cos{\beta} z_{k-1}, \sqrt{-1}\sin{\beta} z_{k-1}, $$
$$z_k, \cdots, z_n, 1-\sqrt{1-\cos(2\theta) z^2_1 - \cos(2{\beta})\sum_{j=2}^{k-1} z^2_j-\sum_{j=k}^n z_j^2} \bigg)$$ for $0 < \beta \leq \theta \leq \pi/4$. Here when $k=n+1,$
the components $``z_k, \cdots, z_n"$ is understood to be void.
Fixing $\alpha \in (0, \pi/4)$ and $\beta \in (0, \pi/4)$, we define 


$$I_{2n+2, \theta} =  \bigg( \cos{\theta} z_1, \sqrt{-1}\sin{\theta} z_1, \cos{\beta} z_2, \sqrt{-1}\sin{\beta} z_2, \cdots, \cos{\beta} z_{n}, \sqrt{-1}\sin{\beta} z_n, $$
$$ \frac{\cos\alpha}{\cos(2\alpha)} \left(1 - \sqrt{1-\cos(2\alpha) \left(\cos(2\theta) z^2_1 + \cos(2{\beta})\sum_{j=2}^{n} z^2_j\right)}\right), $$
$$\frac{\sqrt{-1}\sin\alpha}{\cos(2\alpha)}\left(  \sqrt{1-\cos(2\alpha) \left(\cos(2\theta) z^2_1 + \cos(2{\beta})\sum_{j=2}^{n} z^2_j\right)} -1 \right) \bigg)$$ for $ \beta \leq \theta \leq \pi/4.$
For $n \geq 2, 2 \leq k \leq n+2$, we will show that $I_{n+k,\theta}$ gives a real parameter family of inequivalent minimum holomorphic isometries from $\mathbb{B}^n$ to $D^{IV}_{n+k}$. More precisely,

\begin{theorem}\label{thmneqv}
Let $n \geq 2$. Then the following statements holds.
\begin{itemize}
\item $\{I_{n+2, \theta}\}_{0 < \theta \leq \pi/4}$ is a family of mutually inequivalent minimum  holomorphic isometries.
\item For each $2 < k \leq n+2$ and fixed $\beta \in (0, \pi/4)$, $\{I_{n+k, \theta}\}_{\beta \leq \theta \leq \pi/4}$ is a family of mutually inequivalent minimum holomorphic isometries.
\end{itemize}
More precisely, for each $2 \leq k \leq n+2$, $I_{n+k, \theta_1}$ is equivalent to $I_{n+k, \theta_2}$ if and only if $\theta_1=\theta_2$.
\end{theorem}

\begin{proof}
We will merely prove the case $k=2$ and the remaining cases follow by similar argument. Let $0< \theta_2 < \theta_1 \leq \pi/4$. Then we show that $I_{n+2, \theta_1}$ and $I_{n+2, \theta_2}$ are not equivalent.

Apply the Borel embedding to embed $\mathbb{B}^n$ as an open subset of $\mathbb{P}^n$ and $D^{IV}_{n+2}$ as an open subset of $\mathbb{Q}^{n+2} \subset \mathbb{P}^{n+3}$ as before and write $[z, s] = [z_1, \cdots, z_n, s]$ to denote the homogeneous coordinates in $\mathbb{P}^n$.
Under the homogeneous coordinates, $I_{n+2, \theta}$ can be identified with
$$\mathcal{I}_{n+2, \theta}(z, s)
= \left[\phi_{1, \theta}(z, s), \cdots, \phi_{n+4, \theta}(z, s)  \right]
$$ from $\mathbb{P}^n$ to $\mathbb{P}^{n+3}$, where
\begin{equation}\notag
\begin{split}
\phi_{1, \theta}(z, s) &= \cos\theta z_1; \\
\phi_{2, \theta}(z, s) &= \sqrt{-1} \sin\theta z_1; \\
\phi_{j, \theta}(z, s) &= z_{j-1}  {\rm~for~} 3 \leq j \leq n+1 ;\\
\phi_{n+2, \theta}(z, s) &= s-\sqrt{H_\theta(z, s)}; \\
\phi_{n+3, \theta}(z, s) &= \frac{1}{\sqrt{2}} \left( 2s - \sqrt{H_\theta(z, s)} \right); \\
\phi_{n+4, \theta}(z, s) &= \frac{1}{\sqrt{-2}} \sqrt{H_\theta(z, s)},
\end{split}
\end{equation}
where $H_\theta(z, s) = s^2-\cos(2{\theta}) z^2_1 - \sum_{j=2}^{n} z_j^2$.
Note that for any $\theta \in [0, \pi/4)$ and $n \geq 2$, $H_\theta$ is irreducible and in particular, $H_\theta$ is not a perfect square of a polynomial in $(z, s)$.

By the previous argument, $I_{n+2, \theta_1}$ is equivalent to $I_{n+2, \theta_2}$ if and only if
there exist ${\bf U} \in U(n, 1)$ and $\begin{bmatrix} A & B\\ C & D\end{bmatrix} \in {\rm Aut}(D^{IV}_{n+2})$ such that \begin{equation}
\left(\phi_{1, \theta_1}, \cdots, \phi_{n+2, \theta_1}\right)((z, s) {\bf U}) \equiv \frac{ (\phi_{1, \theta_2}, \cdots, \phi_{n+2, \theta_2}) A + (\phi_{n+3, \theta_2}, \phi_{n+4, \theta_2}) C}{\left((\phi_{1, \theta_2}, \cdots, \phi_{n+2, \theta_2}) B + (\phi_{n+3, \theta_2}, \phi_{n+4, \theta_2}) D\right) (1/\sqrt{2}, \sqrt{-1/2})^t }(z, s).
\end{equation}
Consequently, one has $$\sqrt{H_{\theta_1}((z, s) {\bf U})} = R_1(z, s) \sqrt{H_{\theta_2}(z, s)} + R_2(z, s)$$ for rational functions $R_1(z, s), R_2(z, s)$. This is impossible by algebra if the following claim is true.

\medskip

{\bf Claim:} For any ${\bf U} \in U(n, 1)$, $H_{\theta_1}((z, s) {\bf U})$ and $H_{\theta_2}(z, s)$ are coprime.

\medskip

{\bf Proof of Claim:} Suppose not. Since they are both irreducible, then there exists ${\bf U} \in U(n, 1)$ such that
\begin{equation}\label{qur}
H_{\theta_1}((z, s) {\bf U}) = c H_{\theta_2}(z, s)
\end{equation} for some nonzero complex number $c$. Write $H_\theta = - (z, s) A_\theta (z, s)^t$ with $A_\theta = {\rm diag}(\cos(2\theta), 1, \cdots, 1, -1)$. Then (\ref{qur}) yields that $${\bf U} \cdot A_{\theta_1}\cdot {\bf U}^t = c \cdot A_{\theta_2}.$$ This is impossible by Proposition \ref{imp}. This finishes the proof of the claim.
\end{proof}

\begin{proposition}\label{imp}
Let $\lambda, \lambda_1, \cdots, \lambda_{n+1}$ be real numbers such that $|\lambda| < |\lambda_1| \leq \cdots \leq |\lambda_{n+1}| $ for $n \geq 1$. Then there does not exist an $(n+1)\times (n+1)$ matrix ${\bf U} \in U(n, 1)$, such that
\begin{equation}\label{xxxx}
{\bf U} \cdot {\rm diag}(\lambda, \lambda_2, \cdots, \lambda_{n+1}) \cdot {\bf U}^t = c \cdot {\rm diag}(\lambda_1, \lambda_2, \cdots, \lambda_{n+1})\end{equation} for some complex number $c$.
\end{proposition}

\begin{proof}
Write $${\bf U} = \begin{bmatrix} a_1 & b_1 & c_1 & \cdots \\
a_2 & b_2 & c_2 & \cdots \\
\cdots & \cdots & \cdots & \cdots \\
a_{n+1} & b_{n+1} & c_{n+1} & \cdots  \end{bmatrix}$$ and note that $\{a_1, a_2, \cdots , a_{n+1}\}$ cannot be all zero.
\medskip

{\bf Claim:} Only one element in $\{a_1, \cdots , a_{n+1}\}$ is not zero.

\medskip

{\bf Proof of Claim:} We will merely present the proof for $n=3$ and the general case is  similar. Suppose that the claim is not true. Then any vector $(a_i, a_j, a_k)$ for $1 \leq i < j < k \leq 4$ is a nonzero vector. We now claim
\begin{equation}\label{one}
a_4 \cdot {\rm det}\overline{\begin{bmatrix} a_1 & c_1 & d_1 \\ a_2 & c_2 & d_2 \\ a_3 & c_3 & d_3\end{bmatrix}} = - b_4 \cdot {\rm det}\overline{\begin{bmatrix} b_1 & c_1 & d_1 \\ b_2 & c_2 & d_2 \\ b_3 & c_3 & d_3\end{bmatrix}},
\end{equation}
\begin{equation}\label{two}
a_4 \cdot {\rm det}\overline{\begin{bmatrix} a_1 & b_1 & d_1 \\ a_2 & b_2 & d_2 \\ a_3 & b_3 & d_3\end{bmatrix}} = c_4 \cdot {\rm det}\overline{\begin{bmatrix} b_1 & c_1 & d_1 \\ b_2 & c_2 & d_2 \\ b_3 & c_3 & d_3\end{bmatrix}}
\end{equation}
\begin{equation}\label{three}
a_4 \cdot {\rm det}\overline{\begin{bmatrix} a_1 & b_1 & c_1 \\ a_2 & b_2 & c_2 \\ a_3 & b_3 & c_3\end{bmatrix}} = d_4 \cdot {\rm det}\overline{\begin{bmatrix} b_1 & c_1 & d_1 \\ b_2 & c_2 & d_2 \\ b_3 & c_3 & d_3\end{bmatrix}}
\end{equation}
We only prove (\ref{one}) and two others are similar. Note if both ${\rm det}\begin{bmatrix} a_1 & c_1 & d_1 \\ a_2 & c_2 & d_2 \\ a_3 & c_3 & d_3\end{bmatrix}$ and ${\rm det}\begin{bmatrix} b_1 & c_1 & d_1 \\ b_2 & c_2 & d_2 \\ b_3 & c_3 & d_3\end{bmatrix}$ are zero, then (\ref{one}) holds trivially. Without loss of generality, assume ${\rm det}\begin{bmatrix} a_1 & c_1 & d_1 \\ a_2 & c_2 & d_2 \\ a_3 & c_3 & d_3\end{bmatrix} \not= 0$. The case when ${\rm det}\begin{bmatrix} b_1 & c_1 & d_1 \\ b_2 & c_2 & d_2 \\ b_3 & c_3 & d_3\end{bmatrix} \not = 0$ can be proved similarly. It follows from ${\bf U} \in U(n, 1)$ that
$$(a_4, b_4, c_4, -d_4) \begin{bmatrix}  \bar a_1 & \bar a_2 & \bar a_3 \\  \bar b_1 & \bar b_2 & \bar b_3\\ \bar c_1 & \bar c_2 & \bar c_3  \\ \bar d_1 & \bar d_2 & \bar d_3 \end{bmatrix} = (0, 0, 0).$$
This implies that $$\begin{bmatrix} \bar a_1 & \bar b_1 & \bar c_1 & \bar d_1 \\ \bar a_2 & \bar b_2 & \bar c_2 & \bar d_2 \\ \bar a_3 & \bar b_3 & \bar c_3 & \bar d_3\\ 0 & 1 & 0 & 0 \end{bmatrix} \begin{bmatrix} a_4\\ b_4\\ c_4 \\ -d_4 \end{bmatrix} = \begin{bmatrix} 0\\0\\0\\b_4  \end{bmatrix}.$$ Namely, $(a_4, b_4, c_4, -d_4)^t$ is the solution of the linear system:
$$\begin{bmatrix} \bar a_1 & \bar b_1 & \bar c_1 & \bar d_1 \\ \bar a_2 & \bar b_2 & \bar c_2 & \bar d_2 \\ \bar a_3 & \bar b_3 & \bar c_3 & \bar d_3\\ 0 & 1 & 0 & 0 \end{bmatrix} \begin{bmatrix} x_1\\ x_2\\ x_3 \\ x_4 \end{bmatrix} = \begin{bmatrix} 0\\0\\0\\b_4  \end{bmatrix}.$$ By the Cramer's rule, we know
$$a_4 = -b_4 \cdot {\rm det} \begin{bmatrix} \bar b_1 & \bar c_1 &\bar d_1 \\ \bar b_2 & \bar c_2 & \bar c_2 \\  \bar b_3 & \bar c_3 & \bar d_3  \end{bmatrix} /  {\rm det} \begin{bmatrix} \bar a_1 & \bar c_1 &\bar d_1 \\ \bar a_2 & \bar c_2 & \bar c_2 \\  \bar a_3 & \bar c_3 & \bar d_3  \end{bmatrix}.$$ This implies (\ref{one}).

We further claim:
\begin{equation}\label{one1}
\frac{\lambda}{\lambda_2} a_4 \cdot {\rm det}\begin{bmatrix} a_1 & c_1 & d_1 \\ a_2 & c_2 & d_2 \\ a_3 & c_3 & d_3 \end{bmatrix}= - b_4 \cdot {\rm det}\begin{bmatrix} b_1 & c_1 & d_1 \\ b_2 & c_2 & d_2 \\ b_3 & c_3 & d_3\end{bmatrix},
\end{equation}
\begin{equation}\label{two2}
\frac{\lambda}{\lambda_3} a_4 \cdot {\rm det}\begin{bmatrix} a_1 & b_1 & d_1 \\ a_2 & b_2 & d_2 \\ a_3 & b_3 & d_3\end{bmatrix} = c_4 \cdot {\rm det}\begin{bmatrix} b_1 & c_1 & d_1 \\ b_2 & c_2 & d_2 \\ b_3 & c_3 & d_3\end{bmatrix}
\end{equation}
\begin{equation}\label{three3}
\frac{\lambda}{\lambda_4} a_4 \cdot {\rm det}\begin{bmatrix} a_1 & b_1 & c_1 \\ a_2 & b_2 & c_2 \\ a_3 & b_3 & c_3\end{bmatrix} = - d_4 \cdot {\rm det}\begin{bmatrix} b_1 & c_1 & d_1 \\ b_2 & c_2 & d_2 \\ b_3 & c_3 & d_3\end{bmatrix}
\end{equation}
We  only prove (\ref{one1}) and two others are similar. Note again if both ${\rm det}\begin{bmatrix} a_1 & c_1 & d_1 \\ a_2 & c_2 & d_2 \\ a_3 & c_3 & d_3\end{bmatrix}$ and ${\rm det}\begin{bmatrix} b_1 & c_1 & d_1 \\ b_2 & c_2 & d_2 \\ b_3 & c_3 & d_3\end{bmatrix}$ are zero, then (\ref{two}) holds trivially. Without loss of generality, assume ${\rm det}\begin{bmatrix} a_1 & c_1 & d_1 \\ a_2 & c_2 & d_2 \\ a_3 & c_3 & d_3\end{bmatrix} \not= 0$. The case ${\rm det}\begin{bmatrix} b_1 & c_1 & d_1 \\ b_2 & c_2 & d_2 \\ b_3 & c_3 & d_3\end{bmatrix} \not = 0$ can be proved similarly. It follows from (\ref{xxxx}) that
$$(\lambda a_4, \lambda_2 b_4, \lambda_3 c_4, \lambda_4 d_4) \begin{bmatrix}  a_1 & a_2 &  a_3 \\   b_1 &  b_2 & b_3\\ c_1 &  c_2 &  c_3  \\ d_1 &  d_2 &  d_3 \end{bmatrix} = (0, 0, 0).$$
This implies:
$$\begin{bmatrix} a_1 &  b_1 & c_1 &  d_1 \\  a_2 &  b_2 & c_2 &  d_2 \\ a_3 &  b_3 &  c_3 &  d_3\\ 0 & 1 & 0 & 0 \end{bmatrix} \begin{bmatrix} \lambda a_4\\ \lambda_2 b_4\\ \lambda_3 c_4 \\ \lambda_4 d_4 \end{bmatrix} = \begin{bmatrix} 0\\0\\0\\ \lambda_2 b_4  \end{bmatrix}.$$ Hence, (\ref{one1}) follows from the Cramer's rule.

Equations (\ref{one}) and (\ref{one1}) imply that $a_4 \cdot {\rm det}\overline{\begin{bmatrix} a_1 & c_1 & d_1 \\ a_2 & c_2 & d_2 \\ a_3 & c_3 & d_3\end{bmatrix}} $ and $\frac{\lambda}{\lambda_2} a_4 \cdot {\rm det}\begin{bmatrix} a_1 & c_1 & d_1 \\ a_2 & c_2 & d_2 \\ a_3 & c_3 & d_3 \end{bmatrix}$ have the same norm. However, $|\lambda / \lambda_2 | <1$. If follows  that
\begin{equation}\label{four}
a_4 \cdot {\rm det}\begin{bmatrix} a_1 & c_1 & d_1 \\ a_2 & c_2 & d_2 \\ a_3 & c_3 & d_3 \end{bmatrix} =0.
\end{equation}
Similarly, (\ref{two}), (\ref{two2}) imply
\begin{equation}\label{five}
a_4 \cdot {\rm det}\begin{bmatrix} a_1 & b_1 & d_1 \\ a_2 & b_2 & d_2 \\ a_3 & b_3 & d_3\end{bmatrix}  =0
\end{equation} and (\ref{three}), (\ref{three3}) imply that
\begin{equation}\label{six}
a_4 \cdot {\rm det}\begin{bmatrix} a_1 & b_1 & c_1 \\ a_2 & b_2 & c_2 \\ a_3 & b_3 & c_3\end{bmatrix} = 0.\end{equation}
Note that $\begin{bmatrix} a_1 & b_1 & c_1 & d_1 \\ a_2 & b_2 & c_2 & d_2 \\ a_3 & b_3 & c_3 & d_3     \end{bmatrix}$ has rank 3 and $(a_1 , a_2, a_3)^t$ is not zero. Then ${\rm det}\begin{bmatrix} a_1 & c_1 & d_1 \\ a_2 & c_2 & d_2 \\ a_3 & c_3 & d_3 \end{bmatrix}$, ${\rm det}\begin{bmatrix} a_1 & b_1 & d_1 \\ a_2 & b_2 & d_2 \\ a_3 & b_3 & d_3\end{bmatrix} $ and ${\rm det}\begin{bmatrix} a_1 & b_1 & c_1 \\ a_2 & b_2 & c_2 \\ a_3 & b_3 & c_3\end{bmatrix}$ cannot be all zero. This together with (\ref{four})-(\ref{six}) implies that $a_4=0.$
Similar argument will yield $a_1 = a_2 = a_3=0$. This is a contradiction and the claim is thus proved.

\medskip

We now assume that $a_{j_0} \not= 0$ for some $1 \leq j_0 \leq n+1$ and all other $a_j =0$. It follows from ${\bf U} \in U(n, 1)$ that $|a_{j_0}|=1$. Write $a_{j_0} = e^{\sqrt{-1}\theta}$ for $\theta \in [0, 2\pi)$. Note ${\bf U} \in U(n, 1)$ implies ${\bf \overline{U}^t} \in U(n,1).$ Write ${\bf u}_i$ as the $i^{\text{th}}$ column of ${\bf U}.~{\bf \overline{U}^t} \in U(n,1)$ implies
\begin{equation}\label{eqnuij}
{\bf \overline{u}}_i \cdot \mathrm{diag}(1, \cdots, 1, -1) \cdot {\bf u}_j=0,~\text{if}~j \neq 1.
\end{equation}

We conclude from (\ref{eqnuij}) the $j_0$-th row of $\bf U$ is $(e^{\sqrt{-1}\theta}, 0, \cdots, 0).$ Interchange the first and $j_0$-th row of $\bf U$ and still denote the new matrix by $\bf U$.
Hence one has \begin{equation}\label{U}
{\bf U} = \begin{bmatrix} e^{\sqrt{-1}\theta} & {\bf 0}_{1 \times n} \\  {\bf 0}_{1 \times n}^t & {\bf V}  \end{bmatrix}
\end{equation}
and \begin{equation}\label{U2}
{\bf U}\cdot {\rm diag}(\lambda, \lambda_2, \cdots, \lambda_{n+1}) \cdot {\bf U}^t=c \cdot{\rm diag}(\lambda_{j_0}, \lambda_{j_1}, \cdots, \lambda_{j_{n}}),
\end{equation}
where $\{j_1, \cdots, j_n \}$ is a permutation of $\{1, \cdots, n+1\} \setminus \{j_0\}$.
It follows from (\ref{U}), (\ref{U2}) that
\begin{equation}\label{U3}
e^{2\sqrt{-1}\theta} \lambda = c \lambda_{j_0}
\end{equation}
and \begin{equation}\label{U4}
{\bf V}\cdot {\rm diag}(\lambda_2, \cdots, \lambda_{n+1}) \cdot {\bf V}^t= c \cdot {\rm diag}( \lambda_{j_1}, \cdots, \lambda_{j_{n}}).
\end{equation}
Recall that $|\lambda| < |\lambda_1| \leq \cdots \leq |\lambda_{n+1}| $. (\ref{U3}) implies that $|c| <1$. Note det$({\bf V}) =1$ as det$({\bf U})=1$. Therefore (\ref{U4}) implies
$$|c|^n = \prod_{j=2}^{n+1} |\lambda_j | / \prod_{k=1}^n|\lambda_{j_k}|  \geq 1.$$ This is a contradiction and thus the proposition is proved.
\end{proof}

\bigskip

\begin{remark}
By a similar argument as in the the proof of Theorem \ref{thmneqv}, one can show that for any $2 \leq k \leq n+2, (I_{n+k, \theta}, {\bf 0}),$ where $\theta$ varies in the given interval,  is a family  of mutually inequivalent holomorphic isometries from $\mathbb{B}^n$
to $D_{m}^{IV}, m \geq n+k.$
\end{remark}

\section{Degree estimates}
In this section, we prove various degree estimate results for holomorphic isometric maps from $\mathbb{B}^n$ to $D_{m}^{IV}.$ We first introduce the following definition.

\begin{definition}
Let $F$ be a rational map from $\mathbb{C}^{n}$ into
$\mathbb{C}^{m}.$ We write
$$F=\frac{(P_{1},\cdots,P_{m})}{R}$$
where $P_{j}, j=1,\cdots,m$, and $R$ are holomorphic polynomials and $F$ is reduced to the lowest order term. The degree of $F,$ denoted by
$\mathrm{deg}(F),$ is defined to be
$$\mathrm{deg}(F): =max\{\mathrm{deg}(P_{1}), \cdots, \mathrm{deg}(P_{m}),\mathrm{deg}(R)\}.$$
\end{definition}

\begin{theorem}\label{thm1}
Assume $m \geq n+1 \geq  3$.
Let $F: \mathbb{B}^n \rightarrow D^{IV}_m$ be a rational holomorphic isometric embedding satisfying $F(0)=0.$
Then deg$(F) \leq 2$. More precisely,  $F$ is either a linear map or deg$(F)=2$.
\end{theorem}

\begin{proof}
Write $F=(f_1, \cdots, f_m)$ and $h=
\frac{\sum_{i=1}^{m} f_i^2}{2}$. It follows from the isometry assumption that
\begin{equation}\label{eqtnorm}
\sum_{j=1}^m |f_j(z)|^2 = \sum_{j=1}^n |z_j|^2 + |h(z)|^2.
\end{equation}
By a lemma of D'Angelo \cite{D2}, there exists a unitary matrix ${\bf U} = (u_{ij}) \in M(m, m; \mathbb{C})$ such that
\begin{equation}\label{eqt6}
\left(\frac{1}{2}\sum_{j=1}^m f^2_j(z), z_1, \cdots, z_n, 0, \cdots, 0 \right) \cdot {\bf U} = (f_1(z), \cdots, f_m(z)).
\end{equation}
Equation (\ref{eqt6}) reads
\begin{equation}\label{albert1}
f_j(z) = u_{1 j} h(z) + \sum_{i=1}^n  u_{i+1, j} z_i
\end{equation} for all $1 \leq j \leq m$. Take the sum of square of the above equations for all $j$, we conclude that
\begin{equation}\label{quadratic}
\begin{split}
2 h &= \left( \sum_{j=1}^{m} u^2_{1 j}\right) h^2 + 2 h\sum_{j=1}^{m} \left( u_{1 j} \sum_{i=1}^n  u_{i+1, j} z_i \right) +\sum_{j=1}^{m} \left(\sum_{i=1}^n  u_{i+1, j} z_i \right)^2 \\
&= \left( \sum_{j=1}^{m} u^2_{1 j}\right) h^2 + 2 h\sum_{j=1}^{m} \left( u_{1 j} \sum_{i=1}^n  u_{i+1, j} z_i \right) + \sum_{i=1}^n \left( \sum_{j=1}^{m}  u^2_{i+1, j} \right) z^2_i + 2 \sum_{i \not= i'} \left( \sum_{j=1}^{m} u_{i+1, j} u_{i'+1, j} \right) z_i z'_{i}.
\end{split}
\end{equation}
This is a quadratic equation
\begin{equation}\label{quadratic1}
a h^2(z) + (p(z)-2) h(z) + q(z) =0
\end{equation}
 where $a =\sum_{j=1}^{n+1} u^2_{1 j}$ and $p(z)$ and $q(z)$ are, if not identically zero, homogeneous polynomials in $z$ of degree 1 and 2 respectively. Note that $h$ is rational.  We now prove that  $h$ must be a rational function of degree 2 in $z$ or identically 0.

When $a\not=0$, we split into two cases: $q(z) \equiv 0$ and $q(z) \not \equiv 0.$ If $q(z) \not \equiv 0,$ then by the quadratic formula,
$$h(z)=\frac{2-p(z) \pm \sqrt{(p(z)-2)^2-2q(z)}}{2a}.$$
This is a contradiction  as $(p(z)-2)^2-4q(z)$ cannot be a perfect sqaure of a polynomial. If $q(z) \equiv 0,$ then we conclude by (\ref{quadratic1}) that $h(z) \equiv 0$ as $h(0)=0.$ When $a=0$, then $h=\frac{q(z)}{2-p(z)}.$ It must be either identically zero or a rational function of degree 2 in $z$. Here note $p(z)-2$
and $q(z)$ must be coprime if $q(z) \not \equiv 0.$

If $h \equiv 0$, then  $f_j$ are linear polynomials for all $j$ by (\ref{albert1}).
Now assume $h \not\equiv 0.$ This corresponds to $a=0$ and $h(z) =\frac{q(z)}{2 - p(z)}$ for $q \not\equiv 0$. We will show that deg$(F)=2$. If $p \equiv 0$, it is trivially true by (\ref{albert1}). If $p \not \equiv 0$, again by (\ref{albert1}),
$$f_j(z) = \frac{-u_{1j} q(z)+p(z) \left( \sum_{i=1}^n  u_{i+1, j} z_i\right)-2 \left( \sum_{i=1}^n  u_{i+1, j} z_i \right) }{p(z)-2} ,$$ denoted by $\frac{N_j(z)}{p(z)-2}$ for $1\leq j \leq m$. We claim that there exists at least one $j_0$ such that deg$(N_{j_0}(z))=2$. This claim will imply deg$(f_{j_0}(z))=2$ as $q$ cannot be divided by $p-2$. We now give a proof of the claim.
 Suppose deg$(N_j(z))=1$ for all $1 \leq j \leq m$. This is equivalent to
 \begin{equation}\label{al}
 u_{1j} q(z) = p(z) \left( \sum_{i=1}^n  u_{i+1, j} z_i \right),~\text{for all}~1 \leq j \leq m.
\end{equation}
Assume $n \geq 2$. Then the matrix $\begin{bmatrix} u_{21}, \cdots, u_{2 m}\\ \cdots, \cdots, \cdots \\ u_{(n+1)1}, \cdots, u_{(n+1) m} \end{bmatrix}$ is of rank equal to $n \geq 2$. Therefore, there exists $1 \leq j_1 < j_2 \leq m$ such that $\begin{bmatrix} u_{2  j_1} \\ \cdots \\ u_{(n+1) j_1} \end{bmatrix}$ and $\begin{bmatrix} u_{2  j_2} \\ \cdots \\ u_{(n+1) j_2}\end{bmatrix}$ are linearly independent and thus $$ \sum_{i=1}^n  u_{i+1, j_1} z_i  \not\equiv 0, \sum_{i=1}^n  u_{i+1, j_2} z_i \not\equiv 0,  \left( \sum_{i=1}^n  u_{i+1, j_1} z_i, \sum_{i=1}^n  u_{i+1, j_2} z_i \right) =1.$$
It follows from (\ref{al}) that $u_{1 j_1} \not= 0, u_{1 j_2} \not=0$ and moreover, $$q(z) = \frac{p(z)  \left( \sum_{i=1}^n  u_{i+1, j_1} z_i \right) }{u_{1 j_1}} =  \frac{p(z)  \left( \sum_{i=1}^n  u_{i+1, j_2} z_i \right) }{u_{1 j_2}}.$$
This contradicts to the linear independence of $\begin{bmatrix} u_{2  j_1} \\ \cdots \\ u_{(n+1) j_1} \end{bmatrix}$ and $\begin{bmatrix} u_{2  j_2} \\ \cdots \\ u_{(n+1) j_2}\end{bmatrix}$.

Therefore we proved that either deg$(F) =2$ or $F$ is a homogeneous linear polynomial map when $n \geq 2$.
\end{proof}

We have a more precise result for $m < 2n.$
\begin{theorem}\label{thm2}
Assume $3 \leq n+1 \leq m < 2n.$
Let $F: \mathbb{B}^n \rightarrow D^{IV}_m$ be a rational holomorphic isometric embedding satisfying $F(0)=0.$
Then deg$(F)=2.$
\end{theorem}
\begin{proof}
By Theorem \ref{thm1}, we have ${\rm deg}(F)=1$ or $2.$ Hence we just need to show ${\rm deg}(F)$ cannot be $1.$ We will prove by seeking a contradiction. Suppose ${\rm deg}(F)=1.$ By the argument in Theorem \ref{thm1}, each $f_{i}, 1 \leq i \leq m,$
is a homogeneous linear polynomial in $z.$ Then $\sum_{j=1}^m f^2_j(z)$ will be a homogeneous quadratic polynomial in $z.$ By collecting terms of degree $4$ on both sides of (\ref{eqtnorm}), we have $\sum_{j=1}^m f^2_j(z)=0.$ Equation (\ref{eqt6}) is then reduced to,
\begin{equation}\label{eqt7}
\left(z_1, \cdots, z_n, 0, \cdots, 0 \right){\bf V} = \left(f_1(z), \cdots, f_m(z)\right).
\end{equation}
Here ${\bf V}=\left(
                \begin{array}{c}
                  {\bf v_{1}} \\
                  ... \\
                  {\bf v_{m}} \\
                \end{array}
              \right)
$
is an $m \times m$ unitary matrix, ${\bf v_{i}}$ is an $m-$dimensional row vector, $1 \leq i \leq m.$ We rewrite (\ref{eqt7}) as
\begin{equation}\label{eqt8}
\left(z_{1},...,z_{n}\right)\left(
                              \begin{array}{c}
                                {\bf v}_{1} \\
                                ... \\
                                {\bf v}_{n} \\
                              \end{array}
                            \right)=\left(f_{1},...,f_{m}\right).
\end{equation}
The fact that $\sum_{j=1}^m f^2_j(z)=0,$ implies
\begin{equation}\label{eqt9}
{\bf v}_{i} \cdot {\bf v}_{j}=0,~\text{for all}~1 \leq i,j \leq n.
\end{equation}
As ${\bf V}$ is an unitary matrix, we have,
\begin{equation}\label{eqt10}
{\bf v}_{i} \cdot \overline{\bf v}_{j}=0,~\text{for all}~1 \leq i \neq j \leq n;
\end{equation}
\begin{equation}\label{eqtnormvi}
{\bf v}_{i} \cdot \overline{\bf v}_{i}=1, 1 \leq i \leq n.
\end{equation}
It follows from equations (\ref{eqt9}), (\ref{eqt10}) and (\ref{eqtnormvi}) that
\begin{equation}\label{eqtviReIm}
{\rm Re}{\bf v}_{i} \cdot {\rm Im}{\bf v}_{i}=0,  {\rm Re}{\bf v}_{i} \cdot {\rm Re}{\bf v}_{i}={\rm Im}{\bf v}_{i} \cdot {\rm Im}{\bf v}_{i}=\frac{1}{2}, 1 \leq i \leq n.
\end{equation}
\begin{equation}\label{eqtReIm}
{\rm Re}{\bf v}_{i} \cdot {\rm Im}{\bf v}_{j}=0~\text{for all}~1 \leq i,j \leq n.
\end{equation}
We thus get a collection of $2n$ mutually orthogonal nonzero real vectors $\{{\rm Re}{\bf v}_{i}, {\rm Im}{\bf v}_{i}\}_{i=1}^{n}$ in $\mathbb{C}^m.$ This contradicts the assumption that $m < 2n,$ thus establishes Theorem \ref{thm2}.
\end{proof}

\begin{remark}\label{lin}
We remark that the assumption $m<2n$ is optimal in Theorem \ref{thm2}. Indeed, when
$m=2n,$ we have a linear map $F: \mathbb{B}^n \rightarrow D^{IV}_{2n}:$
$$F(z)=(\frac{\sqrt{2}}{2}z_{1}, \frac{\sqrt{-2}}{2}z_{1},...,\frac{\sqrt{2}}{2}z_{n}, \frac{\sqrt{-2}}{2}z_{n}).$$
\end{remark}

Furthermore, we have the following rigidity result for holomorphic rational isometric map of degree one.
\begin{proposition}
Assume $m \geq  2n$ and $n \geq 2$.
Let $F: \mathbb{B}^n \rightarrow D^{IV}_m$ be a rational holomorphic isometric embedding satisfying $F(0)=0.$
Assume that ${\rm deg}(F)=1.$ Then $F$ is a totally geodesic embedding that is isotropically equivalent to
\begin{equation}
\left(\frac{\sqrt{2}}{2}z_{1}, \frac{\sqrt{-2}}{2}z_{1},...,\frac{\sqrt{2}}{2}z_{n}, \frac{\sqrt{-2}}{2}z_{n}, 0,...,0\right).
\end{equation}
\end{proposition}

\begin{proof}
Recall from the proof of Theorem \ref{thm1} and \ref{thm2}, we have if ${\rm deg}(F)=1,$ then
$F$ is a homogeneous linear map. More precisely, there is an $m \times m$ unitary matrix $V=\left(
     \begin{array}{c}
       {\bf v}_{1} \\
       ... \\
       {\bf v}_{m} \\
     \end{array}
   \right)
$ such that equations (\ref{eqt7})-(\ref{eqtReIm}) hold.

We write ${\bf a}_{i}={\rm Re}{\bf v}_{i}, {\bf b}_{i}={\rm Im}{\bf v}_{i}, 1 \leq i \leq n,$  and write the $2n \times m$ matrix,
$${\bf C}=\sqrt{2}\left(
            \begin{array}{c}
              {\bf a}_{1} \\
              {\bf b}_{1} \\
              ... \\
              ... \\
              {\bf a}_{n} \\
              {\bf b}_{n} \\
            \end{array}
          \right).
$$
As a consequence of (\ref{eqtviReIm}), (\ref{eqtReIm}), we have
$${\bf C}{\bf C}^t={\bf I}_{n}.$$
We extend $\{\sqrt{2}{\bf a}_{j}, \sqrt{2}{\bf b}_{j}\}_{j=1}^n$ to an orthonormal basis $\{\sqrt{2}{\bf a}_{1}, \sqrt{2}{\bf b}_{1},...,\sqrt{2}{\bf a}_{n}, \sqrt{2}{\bf b}_{n}, {\bf c}_{2n+1},...,{\bf c}_{m}\}$
of $\mathbb{R}^m.$
We then set $\widetilde{\bf C}$ to be the $m \times m$ matrix,
$$\widetilde{\bf{C}}^t=\left(
                         \begin{array}{ccc}
                         \sqrt{2}{\bf a}_{1}\\
                         \sqrt{2}{\bf b}_{1}\\
                         ...\\
                         ... \\
                         \sqrt{2}{\bf a}_{n} \\
                         \sqrt{2}{\bf b}_{n} \\
                         {\bf c}_{2n+1} \\
                         ... \\
                         {\bf c}_{m} \\
                       \end{array}
                     \right),
$$
then
\begin{equation}\label{eqntildec}
\widetilde{\bf{C}}\widetilde{\bf{C}}^t={\bf I}_{m}.
\end{equation}
That is $\widetilde{\bf{C}} \in O(m, \mathbb{C}).$
We now define
$$\widetilde{F}=F\widetilde{\bf C}^t.$$
Then $\widetilde{F}$ is orthogonal equivalent to $F.$ Moreover,

\begin{equation}
\widetilde{F}=\left(z_{1},...,z_{n}\right)\left(
                              \begin{array}{c}
                                {\bf v}_{1} \\
                                ... \\
                                {\bf v}_{n} \\
                              \end{array}
                            \right)\widetilde{\bf C}^t=
\left(z_{1},...,z_{n}\right)\left(
\begin{array}{c}
{\bf a}_{1}+\sqrt{-1}{\bf b}_{1} \\
... \\
{\bf a}_{n}+\sqrt{-1}{\bf b}_{n} \\
\end{array}
\right)\widetilde{\bf C}^t.
\end{equation}
It then follows from (\ref{eqntildec}) that
$$\widetilde{F}=\left(\frac{\sqrt{2}}{2}z_{1}, \frac{\sqrt{-2}}{2}z_{1},...,\frac{\sqrt{2}}{2}z_{n}, \frac{\sqrt{-2}}{2}z_{n}, 0,...,0 \right).$$
\end{proof}

\section{Holomorphic maps to classical domains}
In this section we construct  proper holomorphic maps  from $\mathbb{B}^n$ to an irreducible classical domains $\Omega$. If $n < n_\Omega$, our construction gives many examples of  non-isometric proper maps. If $n = n_\Omega$, our examples become non-totally geodesic isometric maps. The existence of these non-totally geodesic holomorphic isometries was first discovered by Mok \cite{M6}.

\subsection{Type I domains}
Let $q \geq p$.
We recall that the type I domain is defined by
$$D^I_{p, q} = \{Z \in M(p, q; \mathbb{C}) | I_p -  Z \overline{Z}^t >0 \}$$ and
the Bergman kernel is given by 
$$K(Z, \bar Z) = c_I \left(\det(I_p -  Z \overline{Z}^t) \right)^{-(p+q)},$$
for some constant $c_I$ depending on $p,q.$
The boundary of $D_{p,q}^I$ is contained in 
 $$\{Z \in M(p,q; \mathbb{C})| \det(I_p-Z \overline{Z}^t)=0\}.$$
We will need the following lemma in algebra(cf. \cite{H2}).  We will denote by  $Z(\begin{matrix}
 i_{1} & ... & i_{k} \\
 j_{1} & ... & j_{k}
   \end{matrix}
)$
the determinant of the submatrix of $Z$ formed by its $i_{1}^{\text{th}},...,i_{k}^{\text{th}}$ rows and $j_{1}^{\text{th}},...,j_{k}^{\text{th}}$ columns.

\begin{lemma}\label{lemmazk}
Let $I_{p}$ be the $p \times p$ identity matrix $(p \geq 1),$ Z be a matrix as above.
\begin{equation}
\mathrm{det}( I_{p}- Z\overline{Z^t})= 1 + \sum_{k=1}^p (-1)^k\left( \sum_{1 \leq i_{1}<i_{2}<...< i_{k}\leq p,\\ 1 \leq j_{1}< j_{2}<...< j_{k} \leq q}\left|
 Z(\begin{matrix}
 i_{1} & ... & i_{k} \\
 j_{1} & ... & j_{k}
   \end{matrix}
)\right|^2\right).
\end{equation}
\end{lemma}

Write $z$ as the coordinates in $\mathbb{C}^{n}.$ Let $G(z)$ be a proper holomorphic map from $\mathbb{B}^n$ to $\mathbb{B}^{p+q-1}, q \geq  p \geq 2, p+q-1 \geq n,$ with $G(0)=0.$
Write $G=(g_1,...,g_q, h_2,...,h_p).$ We define a map $H_G$ from $\mathbb{B}^n$ to $M(p,q;\mathbb{C})$ associated to $G$ as follows.
\begin{equation}\label{eqnhgtype1}
H_G(z)=\left(
    \begin{array}{cccc}
      g_{1} & g_{2} & ... & g_{q} \\
      h_{2} & f_{22} & ... & f_{2q} \\
      ... & ... & ... & ... \\
      h_{p} & f_{p2} & ... & f_{pq} \\
    \end{array}
  \right),
\end{equation}
where $$f_{ij}= \frac{h_{i}g_{j}}{g_{1}-1}, 2 \leq i \leq p, 2 \leq j \leq q.$$
We first note that 
$$\left(
    \begin{array}{c}
      f_{2j} \\
      ... \\
      f_{pj} \\
    \end{array}
  \right)=\frac{g_{j}}{g_{1}-1}\left(
                                 \begin{array}{c}
                                   h_{2} \\
                                   ... \\
                                   h_{p} \\
                                 \end{array}
                               \right), 2 \leq j \leq q;$$
$$\Big(f_{i2}, \cdots ,f_{iq}\Big)=\frac{h_i}{g_1-1}\Big( g_2, \cdots,g_q \Big), 2 \leq i \leq p.$$
Consequently, the only submatrices of $H_G$ possibly with nonzero determinant other than single entries are
$$H_G \left(
     \begin{array}{cc}
       1 & k \\
       1 & l \\
     \end{array}
   \right)=f_{kl}, 2 \leq k \leq p, 2 \leq l \leq q.$$
As before,  we denote by $H_G\left( \begin{array}{cc}
       1 & k \\
       1 & l \\
 \end{array} \right)$ the determinant of the submatrix of $H_G$ formed by the $1^{\text{st}}, k^{\text{th}}$ rows and $1^{\text{st}}, l^{\text{th}}$ columns.
 
 Then by Lemma \ref{lemmazk}, we have 
\begin{equation}\label{eqngh}
 \det(I_p-H_G\overline{H_G}^t)=1-\sum_{i=2}^p |h_i|^2 -\sum_{j=1}^q |g_j|^2.
\end{equation}
We claim that $H_G$ maps $\mathbb{B}^n$ to $D_{p,q}^I.$ Indeed, note $H_G(0)=0, $ and $ \det(I_p-H_G\overline{H_G}^t) >0$ in $\mathbb{B}^n$ by equation (\ref{eqngh}). Thus the claim follows easily from the path-connectedness of $\mathbb{B}^n.$ Then  the 
following proposition is a consequence of equation (\ref{eqngh}).
\begin{proposition}
Let $q \geq p \geq 2, p+q \geq n+1.$ $H_G$ defined above  is a proper holomorphic map from $\mathbb{B}^n$ to $D_{p,q}^I.$
\end{proposition}

When $p+q-1 \geq n+1, $ there is a  proper holomorphic map from $\mathbb{B}^n$ to $\mathbb{B}^{p+q-1}$ that does not have a $C^2$-smooth extension up  to any open piece of $\partial \mathbb{B}^n$ (cf. \cite{Do}). In particular, it  is not isometries with respect to Bergman metrics . Let $G$ be such a map. Then we have,
\begin{theorem}
Let $p+q \geq n+2.$ Then there exists a proper holomorphic map from $\mathbb{B}^n$ to $D_{p,q}^I$ that does not extend $C^2$-smoothly up to any open piece of boundary $\partial \mathbb{B}^n$.  In particular,  it  is not isometric.
\end{theorem}

\bigskip

We next consider the case when $n = n_{D_{p,q}^I} =p+q-1.$ Write $z=(z_{1},...,z_{q},w_{2},...,w_{p})$ as the coordinates in $\mathbb{C}^{p+q-1}.$ Let $G(z)=z$ be the identity map from $\mathbb{B}^{p+q-1}$ to $\mathbb{B}^{p+q-1}.$

Let $R^I_{p, q}=H_{G}.$ Namely,
\begin{equation}
R^I_{p, q}=\left(
    \begin{array}{cccc}
      z_{1} & z_{2} & ... & z_{q} \\
      w_{2} & f_{22} & ... & f_{2q} \\
      ... & ... & ... & ... \\
      w_{p} & f_{p2} & ... & f_{pq} \\
    \end{array}
  \right),
\end{equation}
where $f_{ij}=\frac{w_{i}z_{j}}{z_{1}-1}, 2 \leq i \leq p, 2 \leq j \leq q.$ 
 It is then straightforward to verify that $R^I_{p, q}$ is a holomorphic isometry from $\mathbb{B}^{p+q-1}$
to $D^{I}_{p,q}.$ Indeed, by equation (\ref{eqngh}), we have,
$$\det(I_p-R_{p,q}^I\overline{R_{p,q}^I}^t)=1-\sum_{j=1}^q |z_i|^2- \sum_{i=2}^p |w_i|^2.$$

\subsection{Type II domains}
The type II domain is defined by 
$$D^{II}_m = \{Z \in D^I_{m, m} | Z = - Z^t\}$$
and the Bergman kernel is given by
$$ K(z,\bar z) = c_{II} \left( \det(I_m - Z\overline{Z}^t ) \right)^{-(m-1)},$$
for some positive constant $c_{II}$ depending on $m.$ Its boundary is contained in 
$$\{Z \in M(m, m; \mathbb{C})| \det(I_m-Z \overline{Z}^t)=0\}.$$
We will need the following results from algebra.
\begin{lemma}
Let $A=(a_{ij})$ be a $2n \times 2n, n \geq 1,$ skew-symmetric matrix. Then
$$\mathrm{det}(A)=(pf(A))^2.$$
\end{lemma}

Here $\mathrm{pf}(A)$ is a homogenous polynomial in the matrix entries. This polynomial is called the Pfaffian of the matrix $A$ that can be explicitly written as follows.
Let $\Pi$ be the set of all partitions of $\{1,2,...,2n\}$ into pairs without regard to order.  An element $\alpha \in \Pi$ can be written as

$$\alpha=\{ (i_{1}, j_{1}), (i_{2}, j_{2}),...,(i_{n},j_{n})\}$$
with $i_{k}<j_{k}$ and $i_{1}<i_{2}<...<i_{n}.$ Let

$$\pi=\left[\begin{array}{cccccc}
             1 & 2 & 3 & 4 & ... & 2n \\
             i_{1} & j_{1} & i_{2} & j_{2} & ... & j_{n} \\
\end{array}\right]
$$
be the corresponding permutation. Given a partition $\alpha$ as above, define
$$A_{\alpha}=\mathrm{sgn}(\pi)a_{i_{1},j_{1}}a_{i_{2},j_{2}}\cdots a_{i_{n},j_{n}}.$$
The Pfaffian of $A$ is then given by,
$$\mathrm{pf}(A)=\sum_{\alpha \in \Pi} A_{\alpha}.$$
Note that the determinant of an $n \times n$ skew-symmetric matrix for $n$ odd is always zero. The Pfaffian of an $n \times n$ skew-symmetric matrix for $n$ odd is defined to be zero.  

\bigskip

Moreover, we have the following result from algebra. For more details and its proof, see \cite{H2}, \cite{PS}.

\begin{lemma}\label{otwo}
Let $I_{n}$ be the $n \times n$ identity matrix, Z be an $n \times n$ skew-symmetric matrix.  Then
\begin{equation}\label{eqndets}
\mathrm{det}(I_{n}-Z\overline{Z}^t)=\left( 1+ \sum_{1 \leq k \leq n, 2|k} (-1)^{\frac{k}{2}} \left( \sum_{1 \leq i_{1}<...< i_{k}\leq n } \left| Z\left(
                       \begin{array}{ccc}
                         i_{1} & ... & i_{k} \\
                         i_{1} & ... & i_{k} \\
                       \end{array}
                     \right)
 \right| \right) \right)^2.
\end{equation}
Here $``2|k"$ means that $k$ is divisible by $2.$
\end{lemma}

Write $z$ as the coordinates in $\mathbb{C}^n.$ Let $G(z)$ be a proper holomorphic map from $\mathbb{B}^n$ to $\mathbb{B}^{2m-3}, $ where $m$ is an integer with $2m-3 \geq n, m \geq 3.$ Assume $G(0)=0.$ Write $G=(g_{2},...,g_{m},h_{3},...,h_{m}).$ We define a holomorphic  map $H_G$ from $\mathbb{B}^n$ to $M(m,m;\mathbb{C})$ associated to $G$ by
\begin{equation}\label{eqnhgtype2}
H_G(z)=\left(
    \begin{array}{cccccc}
      0 & g_{2} & g_{3} & g_{4} & ... & g_{m} \\
      -g_{2} & 0 & h_{3} & h_{4} & ... & h_{m} \\
      -g_{3} & -h_{3} & 0 & f_{34} & ... & f_{3m} \\
      ... & ... & ... & ... & ... & ... \\
      -g_{m-1} & ... & ... & ... & 0 & f_{(m-1)m} \\
      -g_{m} & -h_{m} & ... & ... & ... & 0 \\
    \end{array}
  \right),
\end{equation}
where  
$$f_{ij}=\frac{g_{i}h_{j}-g_{j}h_{i}}{g_{2}-1}, 3 \leq i,j \leq m.$$

\begin{proposition}
Let $m \geq 3, 2m-3 \geq n.$  $H_G$ in (\ref{eqnhgtype2})  is a proper holomorphic map from $\mathbb{B}^n$ to $D_{m}^{II}.$
\end{proposition}
\begin{proof}
We first compute the determinants of the principal submatrices of $H_G.$ It follows from the straightforward calculation that 
$$H_G \left(
   \begin{array}{cc}
     1 & i \\
     1 & i \\
   \end{array}
 \right)=(g_i)^2, 2 \leq i \leq m;$$
 
$$H_G \left(
   \begin{array}{cc}
     2 & j \\
     2 & j \\
   \end{array}
 \right)=(h_j)^2, 3 \leq j \leq m;$$
 
$$H_G \left(
   \begin{array}{cc}
     i & j \\
     i & j \\
   \end{array}
 \right)=(f_{ij})^2, 3 \leq i < j  \leq m.$$
Moreover, 
\begin{equation}
H_G \left(
   \begin{array}{cccc}
     1 & 2 & k & l \\
     1 & 2 & k & l \\
   \end{array}
 \right)=(g_2 f_{kl}-g_kh_l+g_l h_k)^2=(f_{kl})^2,~~3 \leq k < l \leq m;
\end{equation}

\begin{equation}\label{eqt1f}
H_G \left(
   \begin{array}{cccc}
     1 & j & k & l \\
     1 & j & k & l \\
   \end{array}
 \right)=(g_{j}f_{kl}-g_{k}f_{jl}+g_{l}f_{jk})^2=0,~~3 \leq j < k <l \leq m;
\end{equation}

\begin{equation}\label{eqt2f}
H_G \left(
   \begin{array}{cccc}
     2 & j & k & l \\
     2 & j & k & l \\
   \end{array}
 \right)=(h_{j}f_{kl}-h_{k}f_{jl}+h_{l}f_{jk})^2=0,~~3 \leq j < k <l \leq m;
\end{equation}

\begin{equation}\label{eqtif}
H_G \left(
   \begin{array}{cccc}
     i & j & k & l \\
     i & j & k & l \\
   \end{array}
 \right)=(f_{ij}f_{kl}-f_{ik}f_{jl}+f_{jk}f_{il})^2=0,~~3 \leq i<j<k<l \leq m.
\end{equation}

For higher order submatrices, we have the lemma below. 
\begin{lemma}
Every  principal submatrix of $H_G$ with order $\geq 5$ has zero determinant.
\end{lemma}
\begin{proof}
First we note that any $m \times m$ anti-symmetric matrix for odd $m$ has zero determinant. For $m$ even, we recall the fact that the Pfaffian of an anti-symmetric
$m \times m$ matrix $A$ can be computed recursively as
$${\rm pf}(A)=\sum_{j=2}^m (-1)^j a_{1j} {\rm pf}(A_{\hat{1}\hat{j}}).$$
Here $A_{\hat{1}\hat{j}}$ denotes the matrix obtained from $A$ with both
its $1$-st and $j$-th rows and columns removed. This together with
(\ref{eqt2f})-(\ref{eqtif}) yields that all principal submatrices of $H_G$ with even order
$\geq 6$ has zero determinant.
\end{proof}

 Then it follows from Lemma \ref{otwo} that
\begin{equation}\label{eqnghj}
\det(I_n-H_G\overline{H_G}^t)=\left(1-\sum_{i=2}^m |g_i|^2-\sum_{j=3}^m |h_j|^2\right)^2.
\end{equation}

Therefore we conclude as in type I case that $H_G$ is a proper holomorphic  map from $\mathbb{B}^n$ to $D_{n}^{II}.$
\end{proof}

When $2m-3 \geq n+1,$ i.e., $m \geq 2+ \frac{n}{2},$ there is a  proper holomorphic map from $\mathbb{B}^n$ to $\mathbb{B}^{2m-3}$ that does not have a $C^2$-smooth extension to any open piece of $\partial \mathbb{B}^n$ (See \cite{Do}).  In particular,  it is  not isometric. Let $G$ be such a map.
Then $H_G$ is not isometric, either.  We thus have proved,
\begin{theorem}
Let $m, n$ be  integers such that $n \geq 2, m \geq 2+\frac{n}{2}.$ Then there is a proper holomorphic map from $\mathbb{B}^n$ to $D_m^{II}$ that does not extend $C^2$-smoothly to any open piece of $\partial \mathbb{B}^n.$ In particular, it  is not isometric.
\end{theorem}

Now we consider the case $n=n_{D_m^{II}} =2m-3$ where $m$ is an integer. Write $z=(z_{2},...,z_{m},w_{3},...,w_{m})$ as the coordinates of $\mathbb{C}^{2m-3}, n \geq 5.$ Let $G(z)=z$ be th identity map from
$\mathbb{B}^{2m-3}$ to itself. In this case, $H_G$ is reduced to  $R^{II}_m: \mathbb{B}^{2m-3}\rightarrow D^{II}_{m}$ given by
\begin{equation}
R^{II}_m =\left(
    \begin{array}{cccccc}
      0 & z_{2} & z_{3} & z_{4} & ... & z_{m} \\
      -z_{2} & 0 & w_{3} & w_{4} & ... & w_{m} \\
      -z_{3} & -w_{3} & 0 & f_{34} & ... & f_{3m} \\
      ... & ... & ... & ... & ... & ... \\
      -z_{m-1} & ... & ... & ... & 0 & f_{(m-1)m} \\
      -z_{m} & -w_{m} & ... & ... & ... & 0 \\
    \end{array}
  \right),
\end{equation}
where $f_{ij}=\frac{z_{i}w_{j}-z_{j}w_{i}}{z_{2}-1}, 3 \leq i, j \leq m.$ By equation (\ref{eqnghj}), we have
$$\det(I_n-R_m^{II}\overline{R_m^{II}}^t)=(1-\sum_{i=2}^m |z_i|^2-\sum_{j=3}^m |w_j|^2)^2.$$
We thus conclude that $R^{II}_m$ is a holomorphic isometry from $\mathbb{B}^{2m-3}$ to $D^{II}_{m}.$

\subsection{Type III domains}
The type III dmain is defined by 
$$D^{III}_m = \{Z \in D^I_{m, m} | Z = Z^t\}$$
and the Bergman kernel is given by 
$$K(Z, \bar Z) = c_{III} \left( \det(I_m -  Z\overline{Z}^t) \right)^{-(m+1)}$$
for some constant $c_{III}$ depending on $m.$
Its boundary is contained in
$$\{Z \in M(m, m; \mathbb{C})| \det(I_m-Z \overline{Z}^t)=0\}.$$

We write $z=(z_{1},...,z_{n})$ as the coordinates in $\mathbb{C}^n.$ Let $G(z)$ be a proper holomorphic map from $\mathbb{B}^n$ to $\mathbb{B}^m, m \geq \mathrm{max}\{n,2\},$ with $G(0)=0.$ Write $G=(g_1,...,g_m).$ We define a holomorphic map $H_{G}$ from $\mathbb{B}^n$ to $M(m,m;\mathbb{C})$ associated 
to $G:$
\begin{equation}\label{eqnhgtype3}
H_G(z)=\left(
    \begin{array}{cccc}
      g_{1} & \frac{g_{2}}{\sqrt{2}} & ... & \frac{g_{m}}{\sqrt{2}} \\
      \frac{g_{2}}{\sqrt{2}} & f_{22} & ... & f_{2m} \\
      ... & ... & ... & ... \\
      \frac{g_{m}}{\sqrt{2}} & f_{m2} & ... & f_{mm} \\
    \end{array}
  \right),
\end{equation}
where
$$f_{ij}=f_{ji}=\frac{g_{i}g_{j}}{2(g_{1}-1)}, 2 \leq i,j \leq m. $$
Then we can prove the following proposition similarly.

\begin{proposition}
Let $m \geq  \mathrm{max}\{n ,2\}, H_G$ be as above. Then $H_G$ is a proper holomorphic proper map from $\mathbb{B}^n$ to $D_m^{III}.$
\end{proposition}
\begin{proof}
We first compute determinants of all  $2 \times 2$ submatrices of $H_G.$ 
$$H_G \left(
   \begin{array}{cc}
     1 & i \\
     1 & k \\
   \end{array}
 \right)=f_{ik}, 2 \leq i , k \leq m;$$

$$H_G\left(
                    \begin{array}{cc}
                      1 & i \\
                      k & l \\
                    \end{array}
                  \right)=H_G \left(
                             \begin{array}{cc}
                               k & l \\
                               1 & i \\
                             \end{array}
                           \right)=0,~2 \leq i \leq m, ~2 \leq k <l \leq m;$$
$$H_G\left(
                                         \begin{array}{cc}
                                           i & j \\
                                           k & l \\
                                         \end{array}
                                       \right)=0,~2 \leq i < j \leq m,~2 \leq k < l \leq m.$$
This implies that the determinants of all $3 \times 3$ submatrices of $H_G$ are zero. We then obtain by Lemma \ref{lemmazk} that 
\begin{equation}\label{eqngi}
\det(I_n-H_G\overline{H_G}^t)=1-\sum_{i=1}^m |g_i|^2. 
\end{equation}
Also note $H_G(0)=0.$ We conclude as before that $H_G$ is a proper holomorphic map from $\mathbb{B}^n$ to $D_m^{III}.$                        
\end{proof}

Again since when $m \geq n+1,$ there is a proper holomorphic map $G$ that does not have a $C^2$-smooth extension up to any open piece of $\partial \mathbb{B}^n$ (See \cite{Do}). In particular, it  is not isometric. Then $H_G$ is not an isometry, either. Thus we have
\begin{theorem}
Let $m \geq n+1.$ Then there is a proper holomorphic map from $\mathbb{B}^n$ to $D_m^{III}$ that does not extend $C^2-$smoothly to any open piece of $\partial \mathbb{B}^n$. In particular, it  is not isometric.
\end{theorem}

Next when $n=n_{D^{III}_m}=m,$ we write $G(z)=z$ be the identity map from $\mathbb{B}^m$ to $\mathbb{B}^m.$ Let $R^{III}_m=H_G.$ Namely,
\begin{equation}
R^{III}_n=\left(
    \begin{array}{cccc}
      z_{1} & \frac{z_{2}}{\sqrt{2}} & ... & \frac{z_{m}}{\sqrt{2}} \\
      \frac{z_{2}}{\sqrt{2}} & f_{22} & ... & f_{2m} \\
      ... & ... & ... & ... \\
      \frac{z_{m}}{\sqrt{2}} & f_{m2} & ... & f_{mm} \\
    \end{array}
  \right),
\end{equation}
where $f_{ij}=f_{ji}=\frac{z_{i}z_{j}}{2(z_{1}-1)}, 2 \leq i,j \leq m.$ It is easy to verify that 
 $R^{III}_m$ is a holomorphic isometry from
$\mathbb{B}^m$ to $D^{III}_{m}$ by (\ref{eqngi}).

\subsection{Type IV domains}
The Type IV case was studied in \cite{XY}. Nevertheless we record them here for completeness.
The type IV domain is defined by 
$$D^{IV}_m = \{Z =(z_1, \cdots, z_m) \in \mathbb{C}^m | Z \overline{Z}^t <2 ~\text{and} ~ 1- Z \overline{Z}^t + \frac{1}{4} |Z Z^t|^2 >0 \}$$
 and the Bergman kernel is given by
$$K(Z, \bar Z) = c_{IV} \left( 1 -  Z\overline{Z}^t +\frac{1}{4} |Z Z^t|^2 \right)^{-m}$$
for some positive constant $c_{IV}$ depending on $m.$
Its boundary is given by 
$$ \{Z =(z_1, \cdots, z_m) \in \mathbb{C}^m | Z \overline{Z}^t \leq 2 ~\text{and} ~ 1- Z \overline{Z}^t + \frac{1}{4} |Z Z^t|^2 =0 \}.$$
\medskip

Write $z$ as the coordinates in $\mathbb{C}^n.$ Let $G$ be a proper holomorphic map from $\mathbb{B}^n$ to $\mathbb{B}^m, m \geq n,$ with $G(0)=0.$ Write $G=(g_1,...,g_m).$
We define two holomorphic maps $H_G, W_G$ from $\mathbb{B}^n$ to $\mathbb{C}^{m+1}$ associated to $G:$
\begin{equation}\label{eqnhgtype4}
H_G=(f_{1},...,f_{m-1},f_{m},f_{m+1}),
\end{equation}
where $f_i=g_i, 1 \leq i \leq m-1, f_{m}=\frac{P_{m}}{Q}, f_{m+1}=\frac{P_{m+1}}{Q},$
$$P_{m}=\frac{1}{2}\sum_{i=1}^{m-1}g_{i}^2-g_{m}^2+g_{m}, P_{m+1}=\sqrt{-1}\left(\frac{1}{2}\sum_{i=1}^{m-1}g_{i}^2+g_{m}^2-g_{m}\right), Q=\sqrt{2}(1-g_{m}),$$
and

\begin{equation}\label{eqnwtype4}
W_G=\left(g_1, \cdots, g_{m-1}, g_m, 1-\sqrt{1-\sum_{j=1}^m g_j^2}\right).
\end{equation}
By direct computation we have
\begin{equation}\label{eqnhwg}
1-H_G \overline{H_G}^t +\frac{1}{4}|H_G H_G^t|^2=1-W_G\overline{W_G}^t+ \frac{1}{4} |W_GW_G^t|^2=1-\sum_{i=1}^m |g_i|^2.
\end{equation}
Hence as before we conclude
\begin{proposition}
Let $m \geq n, H_G, W_G$ be as above. Then  $H_G, W_G$ are both proper holomorphic maps from $\mathbb{B}^n$ to $D_{m+1}^{IV}.$
\end{proposition}

Again if $m \geq n+1,$ we can choose $G$ to be  a proper holomorphic map from $\mathbb{B}^n$ to $\mathbb{B}^m$ that does not have a $C^2-$smooth extension to any open piece of $\partial \mathbb{B}^n$(See \cite{Do}).  Then we have,

\begin{theorem}
Let $m \geq n+1.$ Then there is a proper holomorphic map from $\mathbb{B}^n$ to $D_{m+1}^{IV}$ that does  not extend $C^2-$smoothly to any open piece of $\partial \mathbb{B}^n$. In particular, it is not isometric. 
\end{theorem} 

Now when $n=n_{D^{IV}_{m+1}}=m,$ write $z=(z_{1},...,z_{m})$ be the coordinates in $\mathbb{C}^m$ for $m  \geq 2.$ Let $G(z)=z$ be the identity map from $\mathbb{B}^m$ to $\mathbb{B}^m.$ Then $H_G,W_G$ are reduced to $R^{IV}_m, I^{IV}_m: \mathbb{B}^m \rightarrow D^{IV}_{m+1}.$ Here
\begin{equation}\label{stad}
R^{IV}_m=(r_{1},...,r_{m-1},r_{m},r_{m+1}),
\end{equation}
where $r_{i}=z_{i}, 1 \leq i \leq m-1, r_{m}=\frac{p_{m}}{q}, r_{m+1}=\frac{p_{m+1}}{q},$
$$p_{m}=\frac{1}{2}\sum_{i=1}^{m-1}z_{i}^2-z_{m}^2+z_{m}, ~~p_{m+1}=\sqrt{-1}\left(\frac{1}{2}\sum_{i=1}^{m-1}z_{i}^2+z_{m}^2-z_{m}\right),~~ q=\sqrt{2}(1-z_{m}).$$

\begin{equation}\label{eqnstad}
I^{IV}_{m}=\left(z_1, \cdots, z_{m-1}, z_m, 1-\sqrt{1-\sum_{j=1}^m z_j^2}\right).
\end{equation}
It is easy to see from equation (\ref{eqnhwg}) that $R^{IV}_{m}, I^{IV}_{m}$ are both holomorphic isometries. We showed in \cite{XY} that $R^{IV}_m, I^{IV}_{m}$ are the only holomorphic isometries from $\mathbb{B}^{m}$ to $D^{IV}_{m+1}$ up to holomorphic automorphisms. We also proved in \cite{XY} that they are the only two proper holomorphic maps from $\mathbb{B}^m$ to $D_{m+1}^{IV}$
satisfying certain boundary regularity when $m \geq 4.$

\subsection{Singularities of holomorphic isometries}
The rational holomorphic isometries given in previous sections from the unit ball $\mathbb{B}^{n_\Omega}$ into an irreducible classic symmetric domain $\Omega$ are not totally geodesic and only produce singularities at one single point on the boundary $\partial\mathbb{B}^{n_\Omega}$. When $n_\Omega \geq 2$, one can easily avoid passing through this point by slicing $\mathbb{B}^{n_\Omega}$ with a complex hyperplane. Therefore, one obtains holomorphic polynomial isometries from $\mathbb{B}^{n_\Omega-1}$ into $\Omega$. In particular, this answers the question raised by Mok in \cite{M4} (Question 5.2.2) in the negative while the positive answer may still be possible from $\mathbb{B}^{n_\Omega}$ to $\Omega$. Note that these examples are discovered independently by Chan-Mok \cite{CM}.

\begin{theorem}\label{tian}
There exist non-totally geodesic
 holomorphic isometries from the unit ball $\mathbb{B}^{m}$ to the four types of  irreducible bounded symmetric domain that extends holomorphically to $\mathbb{C}^{m}$.
\end{theorem}

\begin{proof}
Holomorphic isometries $F$ are given by the following maps for different targets.
\begin{equation}\notag
\begin{split}F(z_2, \cdots, z_q, w_2, \cdots, w_p) =
\left(
    \begin{array}{cccc}
      0 & z_{2} & ... & z_{q} \\
      w_{2} & -w_2 z_2 & ... & -w_2 z_q \\
      ... & ... & ... & ... \\
      w_{p} & -w_p z_2 & ... & -w_p z_q \\
    \end{array}\right) &: \mathbb{B}^{p+q-2} \rightarrow D^{I}_{p, q} ~{\rm for}~q\geq p\geq 2; \\
    \begin{split}    
   & F(z_3, \cdots, z_n, w_3, \cdots, w_n) = \\
& \left( \begin{array}{cccccc}
      0 & 0 & z_{3} & z_{4} & ... & z_{n} \\
      0 & 0 & w_{3} & w_{4} & ... & w_{n} \\
      -z_{3} & -w_{3} & 0 & w_3 z_4 - z_3 w_4 & ... & z_n w_3 - z_3 w_n \\
      ... & ... & ... & ... & ... & ... \\
      -z_{n-1} & ... & ... & ... & 0 & z_n w_{n-1}- z_{n-1} w_n \\
      -z_{n} & -w_{n} & ... & ... & ... & 0 \\
    \end{array}\right)
    \end{split} 
    &: \mathbb{B}^{2n-4} \rightarrow D^{II}_{n} ~{\rm for}~n\geq 4;\\
    F(z_2, \cdots, z_n)=
 \left(
    \begin{array}{cccc}
      0 & \frac{z_{2}}{\sqrt{2}} & ... & \frac{z_{n}}{\sqrt{2}} \\
      \frac{z_{2}}{\sqrt{2}} & - \frac{z^2_{2}}{2}  & ... & -\frac{z_2 z_n}{2} \\
      ... & ... & ... & ... \\
      \frac{z_{n}}{\sqrt{2}} & -\frac{z_2 z_n}{2} & ... & - \frac{z^2_{n}}{2} \\
    \end{array}
  \right) &: \mathbb{B}^{n-1} \rightarrow D^{III}_n ~{\rm for}~ n\geq 2; \\
  F(z_1, \cdots, z_{n-1})=
  \left(z_1, \cdots, z_{n-1}, -\frac{\sqrt{2}}{4}\sum_{i=1}^{n-1}z_{i}^2, \frac{\sqrt{-2}}{4}\sum_{i=1}^{n-1}z_{i}^2\right) &: \mathbb{B}^{n-1} \rightarrow D^{IV}_{n+1}  ~{\rm for}~ n\geq 2.
\end{split}
\end{equation}
Indeed, these are polynomial holomorphic isometries. 
\end{proof}



\begin{thebibliography}{99}

\bibitem[Al]{Al} Alexander, H.: {\em Holomorphic mappings from the ball and polydisc}, Math. Ann. 209 (1974), 249-256.






\bibitem[C]{C}Calabi, E.: {\em Isometric imbedding of complex manifolds}, Ann.
of Math. (2) 58, (1953). 1--23, MR0057000, Zbl 0051.13103.

\bibitem[CD]{CD} Catlin, D. W. and D'Angelo, J. P.: {\em Positivity conditions for bihomogeneous polynomials}, Math. Res. Lett. 4 (1997), no. 4, 555-567.

\bibitem[Ch]{Ch} Chan, S. T.:
{\em On global rigidity of the $p$-th root embedding},
Proc. Amer. Math. Soc. 144 (2016), no. 1, 347-358.

\bibitem[CM]{CM} Chan, S. T. and Mok, N.: {\em Holomorphic isometric embeddings of complex hyperbolic space forms into irreducible bounded symmetric domains arise from linear sections of minimal embeddings of their compact duals}, preprint

\bibitem[CS]{CS} Cima, J. A. and Suffridge, T. J.: {\em Boundary behavior of rational proper maps}, Duke Math. J.  60  (1990),  no. 1, 135--138, MR1047119, Zbl 0694.32016.

\bibitem[CU]{CU}Clozel, L. and Ullmo, E.: {\em Correspondances modulaires et
mesures invariantes}, J. Reine Angew. Math. 558 (2003), 47--83, MR1979182, Zbl 1042.11027.



\bibitem[D1]{D1} D'Angelo, J. P.: {\em Proper holomorphic maps between balls of different dimensions}, Michigan Math. J. 35 (1988), no. 1, 83-90.

\bibitem[D2]{D2} D'Angelo, J. P.: {Several complex variables and the geometry of real hypersurfaces}, Studies in Advanced Mathematics. CRC Press, Boca Raton, FL, 1993. xiv+272 pp. ISBN: 0-8493-8272-6.


\bibitem[DL1]{DL1} D'Angelo, J. P. and Lebl, J.: {\em Complexity results for CR mappings between spheres}, Internat. J. Math. 20 (2009), no. 2, 149-166. 

\bibitem[DL2]{DL2} D'Angelo, J. P. and Lebl, J.: {\em Homotopy equivalence for proper holomorphic mappings}, Adv. Math. 286 (2016), 160-180.


\bibitem[Do] {Do} Dor, A.: {\em Proper holomorphic maps between balls in one co-dimension}, Ark. Mat. 28 (1990), no. 1, 49-100.

    

\bibitem[Eb1]{Eb1} Ebenfelt, P.: {\em Partial rigidity of degenerate CR embeddings into spheres}, Adv. Math. 239, 72-96 (2013).








\bibitem[Eb2]{Eb2} Ebenfelt, P.: {\em Local Holomorphic Isometries of a Modified Projective Space into a Standard Projective Space; Rational Conformal Factors}, Math. Ann. 363 (2015), no. 3-4, 1333-1348.


\bibitem[Fa]{Fa} Faran, J.: {\em On the linearity of proper maps between balls in the lower codimensional case}, J.
Differential Geom. 24, (1986), 15-17.


\bibitem[FHX]{FHX} Fang, H., Huang, X. and Xiao, M.: {\em Volume-preserving maps between Hermitian symmetric spaces of compact type},  arXiv:1602.01900.


\bibitem[Fo]{Fo} Forstneric, F.: {\em Extending proper holomorphic mappings of positive codimension}, Invent. Math. \textbf{95}, 31-62 (1989), MR0969413, Zbl 0633.32017.


\bibitem[Ha]{Ha} Hamada, H.: {\em Rational proper holomorphic maps from $B^n$ into $B^{2n}$}, Math. Ann. 331 (2005), no. 3, 693-711.

\bibitem[H1]{H1} Hua, L. K.: {\em On the theory of Fuchsian functions of several variables}, Ann. of Math. (2) 47, (1946). 167-191.

\bibitem[H2]{H2} Hua, L. K.: {Harmonic analysis of functions of several complex variables in the classical domains}, (translated from the Russian, which was a translation of the Chinese original) Translations of Mathematical Monographs, 6. American Mathematical Society, Providence, R.I., 1979. iv+186 pp. ISBN: 0-8218-1556-3


\bibitem[Hu1] {Hu1} Huang, X.: {\em On a linearity problem for proper holomorphic maps between 
balls in complex spaces of different dimensions}, J. Differential Geom. 51 (1999), no. 1, 13-33.



\bibitem[Hu2]{Hu2} Huang, X.: {\em On a semi-rigidity property for holomorphic
maps}, Asian J. Math. \textbf{7}(2003), no. 4, 463--492. (A special
issue dedicated to Y. T. Siu on the ocassion of his 60th birthday.) MR2074886, Zbl 1056.32014.

\bibitem[HJ]{HJ} Huang, X. and Ji, S.:  {\em Mapping $B^n$ into $B^{2n-1}$}, Invent. Math. 145(2), 219-250 (2001)


\bibitem[HJX]{HJX} Huang, X., Ji, S. and Xu, D.: {\em A new gap phenomenon for proper holomorphic mappings from $B^n$ into $B^N$}, Math. Res. Lett. 13 (2006), no. 4, 515-529.


\bibitem[HJY1]{HJY1} Huang, X., Ji, S. and Yin, W.: {\em Recent progress on two problems in several complex variables}, Proceedings of International Congress of Chinese Mathematicians 2007, vol. I, pp. 563-575. International
Press (2009).

\bibitem[HJY2]{HJY2} Huang, X., Ji, S. and Yin, W.: {\em On the third gap for proper holomorphic maps between balls}, Math. Ann. 358 (2014), no. 1-2, 115-142.

\bibitem[HY1]{HY1} Huang, X. and Yuan, Y.: {\em Holomorphic isometry from a K\"ahler manifold into a product of complex projective manifolds}, Geom. Funct. Anal. 24 (2014), no. 3, 854-886.

\bibitem[HY2]{HY2} Huang, X. and Yuan, Y.: {\em Submanifolds of Hermitian symmetric spaces}, Analysis and Geometry, 197-206, Springer Proc. Math. Stat., 127, Springer, 2015.


\bibitem[KZ1]{KZ1} Kim, S.-Y. and Zaitsev, D.: {\em Rigidity of CR maps between Shilov boundaries of bounded symmetric domains}, Invent. Math. 193 (2013), no. 2, 409-437.

\bibitem[KZ2]{KZ2} Kim, S.-Y. and Zaitsev, D.: {\em Rigidity of proper holomorphic maps between bounded symmetric domains}, Math. Ann. 362 (2015), no. 1-2, 639-677.

\bibitem[L]{L} L\o w, E.: {\em Embeddings and proper holomorphic maps of strictly pseudoconvex domains into polydiscs and
balls}, Math. Z. 190 (1985), no. 3, 401-410.



\bibitem[M1]{M1} Mok, N.: {Metric rigidity theorems on Hermitian locally symmetric manifolds}, Series in Pure Mathematics. 6. World Scientific Publishing Co., Inc., Teaneck, NJ, 1989. xiv+278 pp.

\bibitem[M2]{M2} Mok, N.: {\em Local holomorphic isometric embeddings arising
from correspondences in the rank-1 case}, Contemporary trends in
algebraic geometry and algebraic topology (Tianjin, 2000), 155--165,
Nankai Tracts Math., 5, World Sci. Publ., River Edge, NJ, 2002, MR1945359, Zbl 1083.32019.

\bibitem[M3]{M3} Mok, N.: {\em  Nonexistence of proper holomorphic maps between certain classical bounded symmetric domains}, Chin. Ann. Math. Ser. B 29 (2008), no. 2, 135-146.



\bibitem[M4]{M4} Mok, N.: {\em Geometry of holomorphic isometries and related maps between bounded domains}, Geometry and analysis. No. 2, 225-270, Adv. Lect. Math. (ALM), 18, Int. Press, Somerville, MA, 2011.

\bibitem[M5]{M5} Mok, N.: {\em Extension of germs of holomorphic isometries
up to normalizing constants with respect to the Bergman metric}, J. Eur. Math. Soc. 14 (2012), no. 5, 1617-1656.


\bibitem[M6]{M6} Mok, N.: {\em Holomorphic isometries of the complex unit ball into irreducible bounded symmetric domains}, Proc. Amer. Math. Soc. 144 (2016), no. 10, 4515-4525.



\bibitem[MN1]{MN1} Mok, N. and Ng, S.: {\em Second fundamental forms of holomorphic isometries of the Poincar\'e disk into bounded symmetric domains and their boundary behavior along the unit circle}, Sci. China Ser. A 52 (2009), no. 12, 2628-2646.


\bibitem[MN2]{MN2} Mok, N. and Ng, S.: {\em Germs of measure-preserving holomorphic maps from bounded symmetric domains to their Cartesian products}, J. Reine Angew. Math.  669 (2012), 47-73.


\bibitem[Ng1]{Ng1} Ng, S.: {\em On holomorphic isometric embeddings of the unit disk into polydisks}, Proc. Amer. Math. Soc. 138 (2010), 2907-2922, MR2644903, Zbl 1207.32018.

\bibitem[Ng2]{Ng2} Ng, S.: {\em On holomorphic isometric embeddings of the unit n-ball into products of two unit m-balls}, Math. Z. 268 (2011) no. 1-2, 347-354, MR2805439, Zbl 1225.53050.

\bibitem[Ng3]{Ng3} Ng, S.: {\em Holomorphic Double Fibration and the Mapping Problems of Classical Domains},  Int. Math. Res. Not. IMRN 2015, no. 2, 291-324.

\bibitem[P]{P} Poincar\'e, H.: {\em Les fonctions analytiques de deux variables et la repr\'esentation conforme}, Rend. Circ. Mat. Palermo (2) 23 (1907), 185-220.

\bibitem[PS] {PS} Pressley, A. and Segal, G.: {Loop groups}, Oxford Mathematical Monographs. Oxford Science Publications. The Clarendon Press, Oxford University Press, New York, 1986. viii+318 pp. ISBN: 0-19-853535-X



\bibitem[St]{St} Stensones, B.: {\em Proper maps which are Lipschitz up to the boundary}, J. Geom. Anal. 6 (1996), no. 2, 317-339.






\bibitem[Ts]{Ts} Tsai, I.-H.: {\em Rigidity of proper holomorphic maps between symmetric domains}, J. Differential Geom. 37 (1993), 123-160.

\bibitem[Tu1]{Tu1} Tu, Z.: {\em Rigidity of proper holomorphic mappings between equidimensional bounded symmetric domains}, Proc. Amer. Math. Soc. 130 (2002), no. 4, 1035-1042.

\bibitem[Tu2]{Tu2} Tu, Z.: {\em Rigidity of proper holomorphic mappings between nonequidimensional bounded symmetric domains}, Math. Z. 240 (2002), no. 1, 13-35.

\bibitem[TH] {TH} Tumanov, A. and Henkin, G.: {\em Local characterization of holomorphic
automorphisms of classical domains}, Dokl. Akad. Nauk SSSR 267 (1982), 796-799. (Russian)
MR 85b:32048




\bibitem[UWZ]{UWZ} Upmeier, H., Wang, K. and Zhang, G.: {\em Holomorphic isometries from the unit ball into symmetric domains}, arXiv:1603.03639.

\bibitem[W]{W} Webster, S. M.: {\em The rigidity of C-R hypersurfaces in a sphere}, Indiana Univ. Math. J. 28 (1979), 405-416.





\bibitem[XY]{XY} Xiao, M. and Yuan, Y.: {\em Holomorphic maps from the complex unit ball to Type IV classical domains}, arXiv:1606.04806.

\bibitem[Yu]{Yu} Yuan, Y.: {\em On local holomorphic maps preserving invariant (p,p)-forms between bounded symmetric domains}, arXiv:1503.00585.

\bibitem[YZ]{YZ} Yuan, Y. and Zhang, Y.: {\em Rigidity for local holomorphic isometric embeddings from $B^n$ into $B^{N_1} \times ... \times B^{N_m}$ up to conformal factors},  J. Differential Geom. 90 (2012), no. 2, 329-349.
\bigskip

\noindent Ming Xiao, mingxiao@illinois.edu, Department of Mathematics, University of Illinois Urbana-Champaign, IL 61801, USA.

 \noindent Yuan Yuan, yyuan05@syr.edu, Department of Mathematics, Syracuse University, NY 13244, USA.



\end{thebibliography}
\end{document}